\documentclass[11pt]{amsart}
\usepackage{amsmath,amssymb,amscd}

\usepackage{latexsym}
\usepackage[dvips]{graphics}
\usepackage[latin1]{inputenc}
\usepackage{amsthm}
\usepackage{amsfonts}
\usepackage{comment}
\usepackage{asymptote}
\usepackage{graphicx,xy,pstricks}
\input xy
\xyoption{all}

\newcommand{\executeiffilenewer}[3]{%
 \ifnum\pdfstrcmp{\pdffilemoddate{#1}}%
 {\pdffilemoddate{#2}}>0%
 {\immediate\write18{#3}}\fi%
}

\newtheorem{theorem}{Theorem}[section]
\newtheorem{proposition}[theorem]{Proposition}
\newtheorem{lemma}[theorem]{Lemma}

\newtheorem{observation}[theorem]{Observation}

\newtheorem{conj}[theorem]{Conjecture}
\theoremstyle{remark}
\newtheorem{remark}{Remark}[section]
\newtheorem{definition}{Definition}[section]
\newtheorem*{theorem*}{Theorem}

\theoremstyle{plain}      

\newcommand{\Z}{\mathbb{Z}}

\newcommand{\R}{\mathbb{R}}
\newcommand{\N}{\mathbb{N}}
\newcommand{\C}{\mathbb{C}}
\newcommand{\HH}{\mathbb{H}^2}
\newcommand{\HHtrois}{\mathbb{H}^3}

\renewcommand{\hom}{\operatorname{Hom}}
\newcommand{\mcg}{\operatorname{Mod}}

\newcommand{\tr}{\operatorname{Tr}}
\newcommand{\eu}{\operatorname{eu}}

\newcommand{\ad}{\operatorname{Ad}}

\newcommand{\PSL}{\mathrm{PSL}_2(\mathbb{R})}
\newcommand{\PSLC}{\mathrm{PSL}_2(\mathbb{C})}

\newcommand{\PSLtild}{\widetilde{\mathrm{PSL}}_2(\mathbb{R})}
\newcommand{\SLdeuxR}{\mathrm{SL}_2(\mathbb{R})}
\newcommand{\sldeux}{\mathrm{sl}_2(\mathbb{R})}
\newcommand{\pu}{\mathrm{pu}}
\newcommand{\PU}{\mathrm{PU}}
\newcommand{\Span}{\operatorname{Span}}
\newcommand{\Aut}{\operatorname{Aut}}
\newcommand{\Out}{\operatorname{Out}}
\newcommand{\Inn}{\operatorname{Inn}}
\newcommand{\Hom}{\operatorname{Hom}}
\newcommand{\Sp}{\operatorname{Sp}}
\newcommand{\Sex}{\operatorname{Sex}}

\newcommand{\re}{\operatorname{Re}}
\newcommand{\im}{\operatorname{Im}}
\newcommand{\ModSo}{\operatorname{Mod}(S_o)}
\newcommand{\XX}{X^\times}
\newcommand{\Diff}{\operatorname{Diff}}
\newcommand{\croix}{\times}

\title[Hourglass representations]{Six-point configurations in the hyperbolic plane and ergodicity of the mapping class group}

\author{Julien March\'e}
\address{Sorbonne Universit\'es, UPMC Univ.\ Paris 06, Institut de Math\'ematiques de Jussieu-Paris Rive Gauche,
UMR 7586, CNRS, Univ.\ Paris Diderot, Sorbonne Paris Cit\'e, 75005 Paris, France}
\email{julien.marche@imj-prg.fr}
\author{Maxime Wolff}
\address{Sorbonne Universit\'es, UPMC Univ.\ Paris 06, Institut de Math\'ematiques de Jussieu-Paris Rive Gauche,
UMR 7586, CNRS, Univ.\ Paris Diderot, Sorbonne Paris Cit\'e, 75005 Paris, France}
\email{maxime.wolff@imj-prg.fr}
\begin{document}

\begin{abstract}
Let $X$ be the space of isometry classes of ordered sextuples of points in the hyperbolic plane such that the product of the six corresponding rotations of angle $\pi$ is the identity. This space $X$ is closely related to the PSL$_2(\mathbb{R})$-character variety of the genus 2 surface $\Sigma$. In this article we study the topology and the natural symplectic structure on $X$, and we describe the action of the mapping class group of $\Sigma$ on $X$. This completes the classification of the ergodic components of the character variety in genus 2 initiated in \cite{Mod2}. 
\vspace{0.2cm}
\end{abstract}

\maketitle
\section{Introduction and statements}

\subsection{A simple dynamical system}

The hyperbolic plane $\HH$ is naturally identified to the subspace
of $\PSL$ consisting of matrices (up to sign) of trace $0$, via the
map which associates to a point $x\in\HH$ the rotation $s_x$ of angle $\pi$.

In this article we will be studying the following space of
{\em configurations of sextuples},
\[ \Sex = \left\lbrace
(x_1,\ldots,x_6)\in\left(\HH\right)^6\,|\, s_{x_6}\cdots s_{x_1}=1
\right\rbrace, \]
and its quotient $X=\Sex/\PSL$ by the natural diagonal action.

Given a configuration $(x_1,\ldots,x_6)\in\Sex$, and $i\in\{1,\ldots,6\}$,
we may perform a {\em leapfrog move} $L_i$, which consists in replacing
$x_i$ with $x_{i+1}$ (in cyclic notation) and $x_{i+1}$ with
$s_{x_{i+1}}(x_i)$. In this move, $x_i$ ``comes to $x_{i+1}$ and pushes it
by the same motion'': this evokes the leapfrog game played by children,
although it may be more accurate to think of a move in Chinese checkers.
This move leaves invariant the four other points, as well as the
product $s_{x_{i+1}}s_{x_i}$, hence preserves $\Sex$.
We will denote by $\ModSo$ the group generated by these leapfrog moves;
this notation will become clear later.

As a simple example of configuration $(x_1,\ldots,x_6)\in\Sex$ we may
choose $x_1$, $x_3$ and $x_5$ arbitrarily and set $x_2=x_1$, $x_4=x_3$
and $x_6=x_5$. Such a configuration, as well as all the elements of their
$\ModSo$-orbits, will be called {\em pinched} configurations.
If moreover $x_1=x_3=x_5$ we call it a {\em singular configuration}:
these configurations yield the only singular point of $X$. We will
denote by $X^*$ the space of isometry classes of non-singular
configurations, and by $\XX$ the space of
isometry classes of configurations of six points which do not lie
in the same geodesic line in~$\HH$.

It will be a simple observation that $X$ has three connected components;
we will denote by $X_0$ the one containing the isometry class of singular
configurations. We will prove the following statements:

\begin{theorem}\label{SexMinimal}
  Let $x\in X_0$ be an isometry class of non-pinched configurations.
  There is a sequence $(\gamma_n)_{n\geq 0}$ in $\ModSo$
  such that the sequence $([\gamma_n\cdot x])_{n\geq 0}$ converges
  to the isometry class of the singular configurations. The
  sequence is provided by a geometric algorithm.
  \end{theorem}

\begin{theorem}\label{SexErgodique}
  The group $\ModSo$ acts on $X_0$ ergodically.
\end{theorem}

Along the way, we will show that $X_0$ is homeomorphic to a conical
neighbourhood of its singularity and derive its homeomorphic type.

The space $X_0$, being a real algebraic variety, has a natural (Lebesgue)
class of measures for which it makes sense, as in
Theorem \ref{SexErgodique}, to say that the
action of $\ModSo$ is
ergodic. Moreover, $X_0^*$ also has a natural symplectic structure,
related to its interpretation as a character variety.

\subsection{Sextuples and representation spaces}

Let $\Sigma$ be a genus two surface,
let $\Gamma$ denote the fundamental group of $\Sigma$,
and let $X(\Gamma)=\Hom(\Gamma,\PSL)/\PSL$ be the space
of morphisms of $\Gamma$ in $\PSL$ up to conjugacy.
A representation $\rho\colon\Gamma\rightarrow\PSL$ is called
{\em elementary} if it has a finite orbit in $\overline{\HH}$.
Equivalently, $\rho$ is non-elementary if its image
is Zariski-dense in $\PSL$.
We denote by $\XX(\Gamma)$ the space of conjugacy
classes of non-elementary representations. By work of W.~Goldman
\cite{Goldman84} this is a smooth $6$-dimensional symplectic manifold.

Let $\mcg(\Sigma)$ be the mapping class group of $\Sigma$.
By the Dehn-Nielsen-Baer theorem, this group may be viewed as the
quotient $\Out^+(\Gamma)=\Aut^+(\Gamma)/\Inn(\Gamma)$ of
orientation-preserving
automorphisms of $\Gamma$ up to inner automorphisms.
A class $[\varphi]\in\Out^+(\Gamma)$ acts on a conjugacy class
$[\rho]\in X(\Gamma)$ by the formula
$[\varphi]\cdot[\rho]=[\rho\circ\varphi^{-1}]$.
In genus two, $\mcg(\Sigma)$ has a special element, the hyperelliptic
involution, which generates its center.

The Euler class
$\eu\colon X(\Gamma)\rightarrow\lbrace -2,-1,0,1,2\rbrace$
measures the obstruction of lifting the representations
$\Gamma\rightarrow\PSL$ to the
universal cover $\PSLtild$. By work of W.~Goldman \cite{Goldman88},
for each $k\in\lbrace -2,-1,1,2\rbrace$, the set $\XX_k(\Gamma)$
of classes of representations of Euler class~$k$ is connected,
and we proved in \cite{Mod2} that the set $\XX_0(\Gamma)$ splits into two
disjoint open sets $X_0^+(\Gamma)$ and $X_0^-(\Gamma)$,
that we denoted by $\mathcal{M}_+$ and $\mathcal{M}_-$.
The hyperelliptic involution fixes $X_0^+(\Gamma)$ pointwise, whereas it
acts on $X_0^-(\Gamma)$ as the conjugation by orientation-reversing
isometries of the plane.
A consequence of \cite{Mod2}, Proposition~1.2, is that both
$X_0^+(\Gamma)$ and $X_0^-(\Gamma)$ are connected.
We will discuss briefly this connectedness in Section~3.3.

If $\varphi\in\Diff_+(\Sigma)$ represents the hyperelliptic involution,
the quotient $S_o$ of $\Sigma$ by the action of $\varphi$ has
the structure of a spherical orbifold with six points of order~$2$.
Let $\Gamma_o$ denote its orbifold fundamental group.
As $\Gamma_o$ has the natural following presentation
\[\Gamma_o=\langle c_1,c_2,c_3,c_4,c_5,c_6 \, | \, c_i^2=1, c_1\cdots c_6=1
\rangle,\]
there is an obvious identification between $\Sex$ and
the space $\Hom'(\Gamma_o,\PSL)$ of morphisms which do not kill
any of the $c_i$'s, hence the space $X$ 
is in bijection with the character variety $X(\Gamma_o)=\Hom'(\Gamma_o,\PSL)/\PSL$.

Now if $\pi\colon\Sigma\rightarrow S_o$ is the quotient by the
action of $\varphi$, the natural map
$\pi_*\colon\Gamma\rightarrow\Gamma_o$
induces a map $\pi^*\colon X(\Gamma_o)\rightarrow X(\Gamma)$,
which restricts to a canonical identification between
$\XX_0$ and $X_0^+(\Gamma)$.
Furthermore, this identification is equivariant for the action of
the group $\mcg(\Gamma_o)$ of leapfrog moves on $\XX_0$, and the
action of $\mcg(\Gamma)$ on~$X_0^+(\Gamma)$.

\subsection{Dynamics of the mapping class group on $\PSL$-characters in genus two}

In \cite{Mod2} we studied the dynamics of $\mcg(\Sigma)$ on $\XX(\Gamma)$,
leaving behind the component $X_0^+(\Gamma)$. Namely, we proved that
$\mcg(\Sigma)$ acts ergodically on each of the components
$\XX_{-1}(\Gamma)$, $\XX_1(\Gamma)$ and $X_0^-(\Gamma)$,
and proved the related result that
every representation in these connected components sends some simple
closed curve to a non-hyperbolic element of $\PSL$.
The proof of the ergodicity in \cite{Mod2} is strongly related to the
existence of {\em non-separating} simple closed curves mapped to
non-hyperbolic elements, which we proved for almost every representation
in these components. By Proposition~1.2 of \cite{Mod2}, the same
technique cannot be applied to representations in $X_0^+(\Gamma)$.

An easy consequence of Theorem~\ref{SexMinimal} is that
every representation in $X_0^+(\Gamma)$ sends some
separating simple closed curve either to the identity or to an elliptic
element of $\PSL$. Then the techniques for proving the ergodicity
of $\mcg(\Sigma)$ are more involved than in \cite{Mod2}.
Together with the results of \cite{Mod2}, Theorems~\ref{SexMinimal}
and \ref{SexErgodique} yield the following statements:

\begin{theorem}\label{BowditchGenre2}
  Let $\rho\colon\Gamma\rightarrow\PSL$ be a representation mapping
  every {\em simple} closed curve to a hyperbolic element. Then
  $\rho$ is faithful and discrete.
\end{theorem}

\begin{theorem}\label{GoldmanGenre2}
  The mapping class group $\mcg(\Sigma)$ acts ergodically on each
  connected component of non-extremal Euler class of $\XX(\Gamma)$.
\end{theorem}

Theorem~\ref{BowditchGenre2} gives an affirmative answer to a question
of B.~Bowditch (see \cite{Bowditch}, question~C) in the genus two case, while
Theorem~\ref{GoldmanGenre2} proves a conjecture of W.~Goldman in the
genus two case.

If a representation $\rho\colon\Gamma\rightarrow\PSL$ sends a separating simple
curve to an elliptic element, we may think of the restriction of
$\rho$ on the fundamental group of each of the two one-holed tori
as the holonomy of a conic hyperbolic structure on a torus with one cone
point.
Thus we may think geometrically of
a generic representation in $X_0^+(\Gamma)$ as two such tori glued along their cone
points. For this reason, we like to call {\em hourglass} these representations.

\subsection{Brief outline of the proofs}

The dynamical system of sextuples of points in $\HH$ acted on by leapfrog
moves is simple enough to find, for every possible non-pinched configuration,
an explicit sequence of leapfrog moves
which decreases the sum $\sum_{i=1}^6d(x_i,x_{i+1})$. This is done
case by case, and leads to the proof of Theorem \ref{SexMinimal}.
This also yields a geometric algorithm which, given any non-elementary
representation of $\Gamma_o$ in $\PSL$, decides whether it is
discrete; thus extending the results of \cite{Gilman} and \cite{GilmanMaskit}
to the group $\Gamma_o$.

Theorem \ref{SexMinimal} enables to reduce the proof of Theorem \ref{SexErgodique} to a
neighbourhood of the singular representation as in \cite{FunarMarche}. This neighbourhood has several
natural, simple and useful interpretations. First, as a set of limits of
sextuple configurations in $\HH$, it may be thought of as a set of
configurations of six points in the Euclidean plane, satisfying extra conditions.
Second, following \cite{Weil64} or \cite{Goldman84}, it can be interpreted in terms of the first
cohomology group of $\Gamma_o$ in $\sldeux$ with coefficients twisted
by the adjoint action of the singular representation.
This leads to a third interpretation as an open set in
the cotangent bundle of the Grassmannian of Lagrangians in the 
symplectic vector space $H_1(\Sigma,\R)$. Each of these
three models bears a natural symplectic structure, and we prove that
the natural symplectic structure on $X_0^+(\Gamma)$ converges to the
relevant natural symplectic form on each model, at the singular
representation.

The idea is then to use the Dehn twists along the separating curves
which are mapped to elliptic elements.
The strategy of the proof, as in \cite{GoldmanXia11}, is to prove
that if $[\rho]$ is sufficiently
close to the singular class of representations, the corresponding
twist flows are transitive on a neighbourhood of $[\rho]$.
In the situation at hand, we do not show whether these twist flows
generate the space of all directions around our representations
(contrarily to \cite{GoldmanXia11} or \cite{Mod2}),
but by using the third model we prove that their directions
generate a completely non-integrable distribution of directions,
hence these flows are indeed transitive; this leads to the proof of
Theorem \ref{SexErgodique}.

\subsection{Organisation of the article}
We introduce some notation in Section \ref{SectionSex} and relate our simple dynamical system
to the dynamics of the mapping class group in genus 2. In Section \ref{SectionAlgo}, we prove
Theorem \ref{BowditchGenre2}. Section \ref{neighbourhood} is devoted to the neighborhood of
the singular configuration whereas Section \ref{SectionDyn} contains the proof of Theorem \ref{GoldmanGenre2}.

\subsection{Acknowledgements}
This work was partially supported by the french ANR ModGroup ANR- 11-BS01-0020 and SGT ANR-11-BS01-0018.
The second author acknowledges support from U.S. National Science Foundation grants DMS 1107452, 1107263, 1107367 ``RNMS: Geometric Structures and Representation Varieties'' (the GEAR Network). We would like to thank Tian Yang and Dick Canary for their kind interest.

\section{Configurations of sextuples}\label{SectionSex}

The aim of this section is to expand on the relation, mentionned
in the introduction, between the space of sextuple configurations
and the $\PSL$-character variety of the surface of genus two.
We will first elaborate on the presentation of the
marked groups $\Gamma$ and $\Gamma_o$, in order to see the group
of leapfrog moves as a mapping class group. It is actually isomorphic
to the 6-strands braid group of the sphere.
We will then recall some elementary drawings relating products of
half-turns; these reminders will be useful later on.
We will then expand on the natural map between $X(\Gamma_o)$ and
$X(\Gamma)$,
and finally exhibit a complete list of types of sextuple configurations,
which will be used in the following section.

\subsection{Markings of the groups $\Gamma$ and $\Gamma_o$}\label{SubsecMarkings}

With suitable markings,
the groups $\Gamma$ and $\Gamma_o$ admit the following presentations,
\[\Gamma_o=\langle c_1,c_2,c_3,c_4,c_5,c_6 \, | \, c_i^2=1, c_1\cdots c_6=1
\rangle,\]
\[ \Gamma=\langle a_1,b_1,a_2,b_2\,|\,[a_1,b_1][a_2,b_2]=1\rangle, \]
and the morphism $\pi_*$ associated to the quotient
by the hyperelliptic involution is defined as follows:
\begin{equation}\label{PiStar}
\pi_*(a_1)=c_1 c_2, \quad\pi_*(b_1)=c_3 c_2,\quad
\pi_*(a_2)=c_4 c_5, \quad \pi_*(b_2)=c_6 c_5.
\end{equation}
Figure \ref{FigureMarquages} is meant to help the reader with the above
conventions for presenting the groups $\Gamma$ and $\Gamma_o$.
It should be noted here that, since the product in a fundamental group
uses concatenation of paths, words in these groups are to be
read from left to right, and our convention for the
commutator here is:
$[a,b]=aba^{-1}b^{-1}$.
\begin{figure}[hbt]
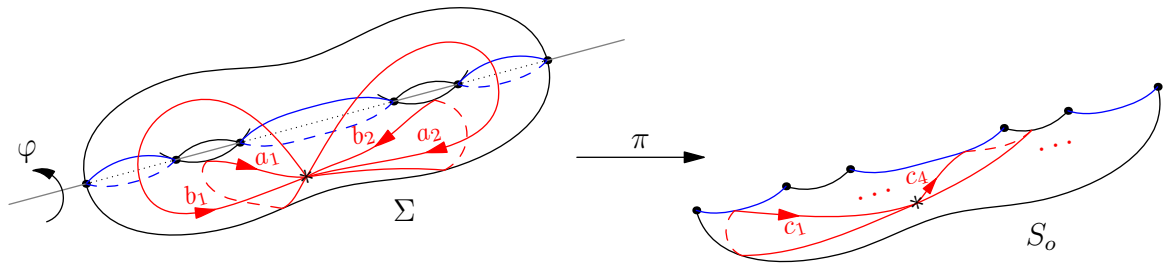

\begin{center}
\begin{asy}
  import geometry;

  picture Sigma;
  //
  //
  point p1 = (0,0), p2 = (35,0), p3 = (60, 0);
  point p4 = (120, 0), p5 = (145,0), p6 = (180,0);

  path contourbas = p1{down}..(50,-30){right}..(90,-26){right}..(130,-30){right}..(180,0){up};
  path contourhaut= p1{up}..(45,30){right}..(90,24){right}..(135,30){right}..(180,0){down};

  dot (Sigma,p1); dot(Sigma,p2); dot(Sigma,p3);
  dot(Sigma,p4); dot(Sigma,p5); dot(Sigma,p6);
  draw(Sigma,contourbas);
  draw(Sigma,contourhaut);

  //
  path hautroug = p2{dir(33)}..p3{dir(-33)};
  path batroug = (p2+(-5,4))..p2..p3..(p3+(5,4));
  draw(Sigma,hautroug); draw(Sigma,batroug);
  path hautroud = p4{dir(33)}..p5{dir(-33)};
  path batroud = (p4+(-5,4))..p4..p5..(p5+(5,4));
  draw(Sigma,hautroud); draw(Sigma,batroud);

  //
  path hautanseg = p1{dir(33)}..p2{dir(-33)};
  path basanseg = p1{dir(-33)}..(17,-4)..p2{dir(33)};
  draw(Sigma,hautanseg,blue); draw(Sigma,basanseg,dashed+blue);
  path hautansed = p5{dir(33)}..p6{dir(-33)};
  path basansed = p5{dir(-33)}..(163,-4)..p6{dir(33)};
  draw(Sigma,hautansed,blue); draw(Sigma,basansed,dashed+blue);
  path hautansem = p3{dir(33)}..(90,6)..p4{dir(-33)};
  path basansem = p3{dir(-33)}..(90,-5)..p4{dir(33)};
  draw(Sigma,hautansem,blue); draw(Sigma,basansem,dashed+blue);

  //
  draw(Sigma,(-30,0)--p1,grey); draw(Sigma,p1--p2,dotted);
  draw(Sigma,p2--p3,grey); draw(Sigma,p3--p4,dotted);
  draw(Sigma,p4--p5,grey); draw(Sigma,p5--p6,dotted);
  draw(Sigma,p6--(210,0),grey);
  //
  draw(Sigma,(-18,-10){dir(5)}..(-18,10){dir(170)},Arrow);
  label(Sigma,"$\varphi$",(-18,10),N);

  //
  point base = (81,-19);
  label(Sigma,"$*$",base);

  //
  //
  picture So;
  draw(So, contourbas);
  dot (So,p1); dot(So,p2); dot(So,p3); dot(So,p4); dot(So,p5); dot(So,p6);
  path SoTroug=p2{dir(-33)}..p3{dir(33)};
  path SoTroud=p4{dir(-33)}..p5{dir(33)};
  draw(So,SoTroug); draw(So,SoTroud);
  draw(So,basanseg,blue); draw(So,basansed,blue); draw(So,basansem,blue);
  label(So,"$*$",base);

  //
  //
  path b1=base{dir(110)}..(45,20){left}..(20,-4){down}..base{dir(6)};
  draw(Sigma,b1,red,Arrow(Relative(0.78)));
  label(Sigma,"\footnotesize $b_1$",relpoint(b1,0.78),dir(128),red);
  path a2=base{dir(45)}..(140,17){right}..(160,0){down}..base{left};
  draw(Sigma,a2,red,Arrow(Relative(0.76)));
  label(Sigma,"\footnotesize $a_2$",relpoint(a2,0.76),dir(70),red);

  point P1=relpoint(contourbas,0.4);
  point P2=relpoint(SoTroug,0.45);
  path a11=base{dir(180+40)}..P1{dir(180+10)};
  path a12=P1{dir(180-25)}..(0.5*(P1+P2)+(-6,-4))..P2{dir(30)};
  path a13=P2{dir(-10)}..base{dir(-10)};
  draw(Sigma,a11,red);
  draw(Sigma,a12,red+dashed);
  draw(Sigma,a13,red,Arrow(Relative(0.5)));
  label(Sigma,"\footnotesize $a_1$",relpoint(a13,0.5),0.05*dir(43),red);

  point P3=relpoint(contourbas,0.7);
  point P4=relpoint(SoTroud,0.6);
  path b21=base{dir(-7)}..P3{dir(-20)};
  path b22=P3{dir(12)}..(P4+(7,-6))..P4{dir(180-30)};
  path b23=P4{dir(180+14)}..(0.5*(P4+base)+(0,-3.5))..base{dir(180+15)};
  draw(Sigma,b21,red);
  draw(Sigma,b22,red+dashed);
  draw(Sigma,b23,red,Arrow(Relative(0.5)));
  label(Sigma,"\footnotesize $b_2$",relpoint(b23,0.5),0.4*dir(120),red);

  //
  //
  point P5=relpoint(contourbas,0.12);
  point P6=relpoint(basanseg,0.4);

  path q11=base{dir(188)}..P5{dir(165)};
  path q12=P5{dir(140)}..P6{dir(30)};
  path q13=P6{dir(-20)}..base{dir(4)};

  draw(So,q11,red);
  draw(So,q12,red+dashed);
  draw(So,q13,red,Arrow(Relative(0.35)));
  label(So,"\footnotesize $c_1$",relpoint(q13,0.35),0.4*SSW,red);
  label(So,"$\cdots$",base,2.6*dir(145),red);

  point P7=relpoint(basansem,0.72);
  point P8=relpoint(SoTroud,0.4);
  path q41=base{dir(35)}..P7{dir(15)};
  path q42=P7{dir(-10)}..P8{dir(30)};
  path q43=P8{dir(210)}..base{dir(190)};
  draw(So,q41,red,Arrow(Relative(0.44)));
  draw(So,q42,red+dashed);
  draw(So,q43,red);
  label(So,"\footnotesize $c_4$",relpoint(q41,0.45),0.8*dir(160),red);
  label(So,"$\cdots$",P8,1.8*dir(-60),red);

  add(rotate(15)*Sigma);
  add(rotate(15)*So,(230,-10));
  draw((185,10)--(233,10),Arrow);
  label("$\pi$",((185+233)/2,10),N);

  label("$\Sigma$",(120,-10));
  label("$S_o$",(360,-20));
\end{asy}
\end{center}
\caption{Markings of the groups $\Gamma$ and $\Gamma_o$}
\label{FigureMarquages}
\end{figure}
On the other hand, we prefer to think of $\PSL$ as acting on $\HH$
on the left, hence we prefer to read words in $\PSL$ from right to left.
For this reason, we will take the convention that morphisms
$\rho\colon\Gamma\rightarrow\PSL$ should be defined as satisfying
the relation $\rho(\alpha\beta)=\rho(\beta)\rho(\alpha)$ for all
$\alpha,\beta\in\Gamma$.
We will also denote, for $A,B\in\PSL$, $[A,B]=B^{-1}A^{-1}BA$.
This convention is reminiscent of
\cite{Mod2} or~\cite{GalloKapovichMarden}.

Every positive self-diffeomorphism $\psi$ of $\Sigma$
commutes, up to isotopy, with the hyperelliptic involution $\varphi$,
hence descends to a diffeomorphism of the sphere with six marked points.
This defines an isomorphism between the quotient
$\mcg(\Sigma)/[\varphi]$ and the group $B_6(S^2)$, the
6-strands braid group of the sphere. This group
is generated by the ``standard'' generators, often denoted by $\sigma_i$,
as schematised in Figure \ref{FigureLeapfrog}.
\begin{figure}[hbt]
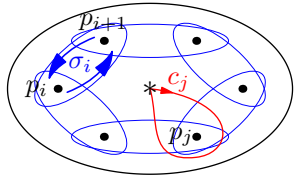

\begin{center}
\begin{asy}
  import geometry;

  unitsize(1pt,0.6pt);

  point base = (0,0);
  real r=35, R=54;
  point P1=r*dir(0), P2=r*dir(60), P3=r*dir(120);
  point P4=r*dir(180), P5=r*dir(240), P6=r*dir(300);
  path cercle=circle(base,R);
  path Cerc1=shift(0.5*(P2+P3))*yscale(0.6)*xscale(1.7)*circle((0,0),r/2);
  path Cerc2=rotate(60,(0,0))*Cerc1, Cerc3=rotate(120,(0,0))*Cerc1;
  path Cerc4=rotate(180,(0,0))*Cerc1, Cerc5=rotate(240,(0,0))*Cerc1;
  path Cerc6=rotate(300,(0,0))*Cerc1;

  path Fleche1=yscale(1.55)*xscale(1.55)*arc((0,0),r/1.4,123,170);
  path Fleche2=rotate(180,(0.5*(P3+P4)))*Fleche1;

  path qjrev=base{dir(-80)}..(P6+(0,-13))..(P6+(10,0))..base{dir(170)};
  path qj=reverse(qjrev);

  draw(cercle);
  draw(Cerc1,0.2pt+blue); draw(Cerc2,0.2pt+blue); draw(Cerc3,0.2pt+blue);
  draw(Cerc4,0.2pt+blue); draw(Cerc5,0.2pt+blue); draw(Cerc6,0.2pt+blue);
  dot(P1); dot(P2); dot(P3); dot(P4); dot(P5); dot(P6);
  draw(Fleche1,0.6pt+blue,Arrow);
  draw(Fleche2,0.6pt+blue,Arrow);
  label("$*$",base);
  label("\footnotesize $p_j$",P6,0.3*W);
  draw(qj,red,Arrow(Relative(0.1)));
  label("\footnotesize $c_j$",relpoint(qj,0.1),0.3*N,red);
  label("\footnotesize $p_i$",P4,W);
  label("\footnotesize $p_{i+1}$",P3,N);
  label("\footnotesize $\sigma_i$",(0.5*(P3+P4)),blue);
\end{asy}
\end{center}
\caption{Standard generators of $B_6(S^2)$}
\label{FigureLeapfrog}
\end{figure}
As they are depicted in Figure \ref{FigureLeapfrog} the diffeomorphisms
$\sigma_i$ fix the base point of $S_o$ hence act as automorphisms of
$\Gamma_o$; we can read:
${\sigma_i}_*(c_i)=c_{i+1}$ and
${\sigma_i}_*(c_{i+1})=c_{i+1}^{-1}c_ic_{i+1}=c_{i+1}c_ic_{i+1}$.
Hence, the action of $\sigma_i$ on representations
$\Gamma_o\rightarrow\PSL$ coincides with the action of the
leapfrog move $L_i$.

In addition to the leapfrog moves corresponding to the $\sigma_i$,
we will often use the cyclic permutation
$\sigma_5\sigma_4\cdots\sigma_1$, which acts on sextuples by
permutation, $(x_1,\ldots,x_6)\mapsto(x_2,\ldots,x_1)$,
as well as the ``half-twist'' $(\sigma_1\sigma_2\sigma_1)^2$,
which replaces $c_1$, $c_2$ and $c_3$ by their conjugates by
$c_1c_2c_3$, thus which acts on sextuples by the formula
$(x_1,\ldots,x_6)\mapsto(f(x_1),f(x_2),f(x_3),x_4,x_5,x_6)$
where $f=s_{x_3}s_{x_2}s_{x_1}$. We call it this way because it is
the image, in $\mcg(S_o)$, of the half-Dehn twist along the
separating curve $[a_1,b_1]$.

Let us insist, finally, that the map
$\mcg(\Sigma)/[\varphi]\rightarrow B_6(S^2)=\mcg(S_o)$ is explicit,
and it is easy to translate an explicit sequence of leapfrog moves into
an explicit sequence of Dehn twists on $\Sigma$. Namely, in
the left part of Figure \ref{FigureMarquages}, consider the three
blue closed curves, and the two black curves making the two handles
of $\Sigma$. It is well-known that the five corresponding Dehn twists
generate $\mcg(\Sigma)$.
The five Dehn twists along these five curves, ordered
from left to right, descend respectively to
$\sigma_1$, \ldots, $\sigma_5$ in $\mcg(S_o)$.

\subsection{Products of three half-turns and commutators}\label{SubsecTroiSym}

Let $x_1,x_2,x_3\in\HH$ and let $s_i$, $i=1,2,3$, be the
half-turn around $x_i$.
If $x_1,x_2,x_3$ are on a same line, then $s_1,s_2,s_3$ can be thought
of as isometries of the real line, and the product $s_3s_2s_1$ is a
half-turn around a point easy to spot on this line (see Figure
\ref{FigureProd3Comm}, left).
If $x_1,x_2,x_3$ are not on a same line, then $x_3$ is at some distance
$h$ from the line $(x_1,x_2)$.
Then $s_3s_2s_1=s_3s_2's_1'$, where $s_2'$ is the half-turn around the
closest point, $x_2'$, to $x_3$ on the line $(x_1,x_2)$, and $s_1'$
is a half-turn around the point $x_1'\in(x_1,x_2)$ chosen so that
$s_2s_1=s_2's_1'$.
Now $s_3s_2s_1$ is the composition of two explicit reflections,
$r_2r_1$ (see Figure \ref{FigureProd3Comm}, right).
\begin{figure}[htb]
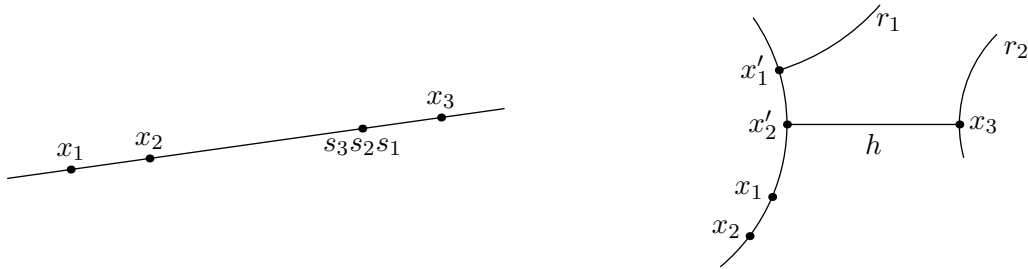

\begin{center}
\begin{asy}
  import geometry;

  point origine1=(0,-17), origine2=(270,0); picture pic, pic2;

  real ecart12=10, ecartmilieu=27, depass=8; pair vecteur = 3*dir(8);
  point p0 = -depass*vecteur, p1 = (0,0), p2 = (ecart12)*vecteur;
  point p3 = (2*ecart12+ecartmilieu)*vecteur;
  point C = (ecart12+ecartmilieu)*vecteur;
  point p5 = (depass+2*ecart12+ecartmilieu)*vecteur;

  draw (pic, p0--p5);
  dot(pic, p1); label(pic, "\small $x_1$", p1, N);
  dot(pic, p2); label(pic, "\small $x_2$", p2, N);
  dot(pic, p3); label(pic, "\small $x_3$", p3, N);
  dot(pic, C); label(pic, "\small $s_3s_2s_1$", C, S);
  add(pic, origine1);

  real hh = 65; point x2p = (0,0), x3 = (hh,0);
  real RayonGauch = 70, RayonDroit = 48, Angl1p = 17;
  point x1p = (-RayonGauch,0)+RayonGauch*dir(Angl1p);
  point x1 = (-RayonGauch,0)+RayonGauch*dir(-23);
  point x2 = (-RayonGauch,0)+RayonGauch*dir(-37);
  point hautgauch = (-RayonGauch,0)+RayonGauch*dir(35);
  point basgauch = (-RayonGauch,0)+RayonGauch*dir(-50);

  point finlignr1 = (35,45); path lignr1 = (x1p{dir(Angl1p)}..finlignr1);

  point hautdroit = (hh+RayonDroit,0)+RayonDroit*dir(135);
  point basdroit = (hh+RayonDroit,0)+RayonDroit*dir(195);

  path ligngauch = (basgauch..x2..x1..x2p..x1p..hautgauch);
  path ligndroit = (hautdroit..x3..basdroit);

  draw (pic2,ligngauch); draw (pic2,ligndroit);
  draw (pic2,lignr1); draw (pic2,x2p--x3);
  dot(pic2,x2p); label(pic2,"\small $x_2'$", x2p, W);
  dot(pic2,x1p); label(pic2,"\small $x_1'$", x1p, W);
  dot(pic2,x1); label(pic2,"\small $x_1$", x1, WNW);
  dot(pic2,x2); label(pic2,"\small $x_2$", x2, WNW);
  dot(pic2,x3); label(pic2,"\small $x_3$", x3, E);
  label(pic2,"\small $r_2$", hautdroit, SE);
  label(pic2,"\small $r_1$", finlignr1, SSE);
  label(pic2,"\small $h$", (hh/2,0),S);
  add(pic2,origine2);
\end{asy}
\end{center}
\caption{Products of three half-turns (left: the centres are
on a line; right: they are in generic position)}
\label{FigureProd3Comm}
\end{figure}
Depending on whether $\sinh(h)\sinh(d(x_1,x_2))$ is less, equal
or greater than $1$, the hyperbolic motion $s_3s_2s_1$ is elliptic,
parabolic or hyperbolic (this follows from the classical formulas in
hyperbolic geometry, see eg \cite{Buser}, page 454, formulas 2.3.1 and 2.3.4).
It is noteworthy that this quantity $\sinh(h)\sinh(d(x_1,x_2))$ is equal to
$\frac{1}{2}\tr(s_3s_2s_1)$ and is invariant under permutations of
the $x_i's$. It is often called the $\Delta$-invariant of the
triangle $x_1,x_2,x_3$.

This can be applied to describe the geometry of commutators of hyperbolic elements
of $\PSL$ whose axes intersect in $\HH$.
If $A$, $B$ are two such isometries,
let $x_2$ be the intersection point of these axes (take $x_2$ to be any
point on this line, if $A$ and $B$ have the same axis).
There exists a unique point $x_1$ on the axis of $A$ such that,
if $s_1$, $s_2$ are the half-turns around $x_1$, $x_2$ we have
$A=s_2s_1$. Similarly, there exists a unique point $x_3$, on the axis of $B$,
such that $B=s_2s_3$, where $s_3$ is the half-turn around $x_3$.
Now,
\begin{equation}\label{TroisSymCommut}
[A,B]=B^{-1}A^{-1}BA=(s_3s_2s_1)^2.
\end{equation}
With this in head, Figure \ref{FigureProd3Comm}
gives a geometric picture of commutators of hyperbolic elements with
crossing axes.


\subsection{Sextuple configurations as representations of $\Gamma$}

The aim of this paragraph is to prove the following correspondence:
\begin{proposition}\label{correspondence}
  The $\PSL$-equivariant map
  $\pi^*\colon \Sex\rightarrow \Hom(\Gamma,\PSL)$
  descends, by restriction, to a homeomorphism 
  $\pi^*\colon \XX(\Gamma_o)\rightarrow X_{-2}(\Gamma)\cup X_0^+(\Gamma)\cup X_{2}(\Gamma)$.
\end{proposition}
In the above statement we write ``homeomorphism'' because we are not
yet concerned with the rich structure of these spaces, but of course
$\pi^*$ carries their structures.

For completeness let us recall the following statement from \cite{Mod2}:
\begin{proposition}[Part of Proposition 1.2 of \cite{Mod2}]\label{Prop12Mod2}
  Let $[\rho]\in X_0^+(\Gamma)$. For every non-separating curve $a$, $\rho(a)$
  is hyperbolic or the identity. For every simple curves $a$, $b$
  such that $i(a,b)=1$, the trace of the commutator
  $[\rho(a),\rho(b)]$ is in $(-\infty,2]$. It is $2$ if and only
  if $[\rho(a),\rho(b)]=1$.
\end{proposition}

\begin{proof}[Proof of Proposition \ref{correspondence}]
  Let $z=(x_1,\ldots,x_6)\in\Sex$, and let $\rho=\pi^*(z)$.
  From Equation (\ref{PiStar}) and the equality $s_{x_6}\cdots s_{x_1}=1$ we compute, formally,
  \[
  \rho([a_1,b_1][a_2,b_2])=(s_{x_6}s_{x_5}s_{x_4})^2(s_{x_3}s_{x_2}s_{x_1})^2
  =(s_{x_6}\cdots s_{x_1})^2. \]
Moreover, the square
  enables to lift $\rho$ to a representation in $\SLdeuxR$;
  this means that $\rho$ has even Euler class.
  Now the representation of $\Gamma_o$ defined by $z$ is elementary
  if and only if $x_1,\ldots,x_6$ are on the same line, and restricting
  a representation to an index two subgroup does not change elementarity.
  This implies that the image of $\XX(\Gamma^o)$ is contained in
  $X_{-2}(\Gamma)\cup X_0^+(\Gamma) \cup X_{2}(\Gamma)$, because, by its construction, the
  conjugacy class of $\rho$ is invariant under the hyperelliptic involution.

  Conversely, let $[\rho]\in X_{-2}(\Gamma)\cup X_0^+(\Gamma) \cup X_{2}(\Gamma)$.
  Suppose for simplicity that $[\rho(a_1),\rho(b_1)]\neq 1$.
  Recall that for $A,B\in\PSL$, the trace of $[A,B]$ is
  well defined, and it is in $(-\infty,2]$ if and only if $A$, $B$
  are hyperbolic and their axes cross each other.
  Then $\rho(a_1)$ and $\rho(b_1)$ are hyperbolic with crossing axes;
  they define three points $x_1$, $x_2$ and $x_3$
  exactly as in the last paragraph of Section \ref{SubsecTroiSym}.
  Similarly, define
  $x_4$, $x_5$ and $x_6$ corresponding to $\rho(a_2)$ and $\rho(b_2)$.
  Then we readily check that $\pi^*(x_1,\ldots,x_6)=\rho$.

  Now if $[\rho(a_1),\rho(b_1)]=1$, we may (for example) first apply an
  explicit automorphism $\psi$ of $\Gamma$ to change
  $(a_1,b_1)$ into $(a,b)$ with
  $a,b\in\{a_1,b_1,a_2,b_2\}$ and $i(a,b)=1$, so that
  $[\rho(a),\rho(b)]\neq 1$ (otherwise $\rho$ would be elementary),
  make the above construction, and come back to the original marking by
  applying $\psi^{-1}$.
  Thus, we have constructed a map
  $X_{-2}(\Gamma)\cup X_0^+(\Gamma) \cup X_{2}(\Gamma)\rightarrow \XX(\Gamma_o)$,
  and we easily check that this map, and $\pi^*$, are the
  inverse to each other.
\end{proof}


\subsection{Different types of sextuple configurations}\label{SubsecTricho}

The aim of this paragraph is to have in head a picture of every
possible configuration of sextuples
$z=(x_1,\ldots,x_6)\in\Sex$, and to set up some notation for the use
of the following section.

In the defining equality we may group the terms two by two,
$(s_{x_6}s_{x_5})(s_{x_4}s_{x_3})(s_{x_2}s_{x_1})=1$. For all
$i\in\{1,\ldots,6\}$ we will write $a_{i,i+1}=d(x_i,x_{i+1})$,
with cyclic notation.
If, say, $x_1=x_2$ then the above relation implies that
$x_3$, $x_4$, $x_5$, $x_6$
are on the same line. In this case, it is easy to deform $z$
among sextuple configurations into a singular configuration.
Now we want to describe the generic configurations, in which
$a_{i,i+1}\neq 0$ for all $i$. We then denote by $D_{i,i+1}$
the line joining $x_i$ and $x_{i+1}$.
Note that $s_{i+1}s_i$ is a hyperbolic translation along that line.

It is elementary and classical to picture the product of
two given hyperbolic motions, say $s_{x_2}s_{x_1}$ and $s_{x_4}s_{x_3}$.
If their axes $D_{12}$ and $D_{34}$ cross each other in $\HH$,
we decompose each of these two motions into two rotations of angle
$\pi$, one of them being around their intersection point (as in
Paragraph \ref{SubsecTroiSym}). The resulting product is the product
of the two other half turns. If $D_{12}$ and $D_{34}$ are disjoint in
$\HH\cup\partial\HH$ we decompose each of the two motions into two
reflections along lines, one of them being the common perpendicular
$H_{23}$ of $D_{12}$ and $D_{34}$. Then $s_{x_6}s_{x_5}$ is the product
of the reflections along the two other lines, which, therefore, cannot
intersect each other in $\HH\cup\partial\HH$. It follows that the
lines $D_{12}$, $D_{34}$ and $D_{56}$, together with the respective
common perpendiculars $H_{23}$, $H_{45}$ and $H_{61}$, form a
right-angled hexagon, which may be regular or skew, depending on whether
the directions of the motions $s_{x_2}s_{x_1}$ and $s_{x_4}s_{x_3}$
agree or disagree (this makes sense for instance by parallel transport
along $H_{23}$). The only remaining case is when $D_{12}$ and $D_{34}$
meet at $\partial\HH$; the motions $s_{x_2}s_{x_1}$ and $s_{x_4}s_{x_3}$
are then contained in a parabolic subgroup of $\PSL$.
The following lemma summarizes the above discussion.
\begin{lemma}\label{tricho}
  Let $z=(x_1,\ldots,x_6)\in\Sex^\times$ be a non-aligned configuration
  with $x_1\neq x_2$, $x_3\neq x_4$ and
  $x_5\neq x_6$. Then one of the following holds.
  \begin{itemize}
  \item[(TRI)] The lines $D_{12}$, $D_{34}$ and $D_{56}$ form a triangle
    with lengths $a_{12}$, $a_{34}$ and $a_{56}$.
  \item[(PAR)] The lines $D_{12}$, $D_{34}$ and $D_{56}$ meet at (at least)
    one common point in $\partial\HH$. In this case, one of the lengths
    $a_{12}$, $a_{34}$ and $a_{56}$ is the sum of the two others.
  \item[(SKH)] The lines $D_{12}$, $D_{34}$ and $D_{56}$ are pairwise disjoint
    in $\overline{\HH}$ and together with their respective common
    perpendiculars they form a skewed right-angled hexagon. In this case,
    one of the lengths $a_{12}$, $a_{34}$ and $a_{56}$ is greater than
    the sum of the two others.
  \item[(HEX)] The lines $D_{12}$, $D_{34}$ and $D_{56}$ are pairwise
    disjoint in $\overline{\HH}$ and with their respective common
    perpendiculars they form a right-angled hexagon.
  \end{itemize}
\end{lemma}

In each of the three situations (TRI), (PAR) and (SKH) it is easy to
deform $z$ among sextuple configurations into a singular configuration;
these cases, together with the degenerate cases above
and the aligned configurations, form the
connected component $X_0$. Of course, we will denote by $\Sex_0$ the
set of corresponding sextuple configurations.
Each case comes with two possible orientations (and a choice of
repelling or attracting point in case (PAR)). In case (HEX), these two
orientations
correspond to the two components of
$X_{\pm 2}(\Gamma)$ via
Proposition~\ref{correspondence}.
\begin{figure}[ht]
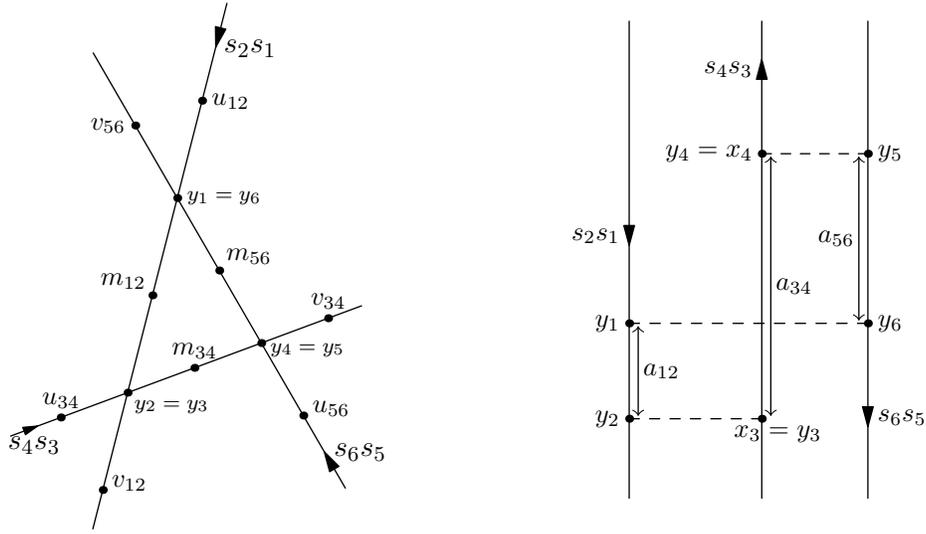

\begin{center}
\begin{asy}
  import geometry;
  unitsize(1pt);

  picture TRI; real redux=0.72;
  pair A1=redux*(13,51), A2=redux*(22,-38), A3=A1+A2;

  // 
  point y3=(0,0), m1=y3+A1, v1=y3-A1, y2=m1+A1, u1=y2+A1;
  dot(TRI,y3); dot(TRI,m1); dot(TRI,v1); dot(TRI,y2); dot(TRI,u1);
  draw(TRI,(u1+A1)--(v1-0.4*A1), Arrow(position=1/12));
  label(TRI,"\small $s_2s_1$",u1+A1/2,dir(12));
  // 
  point v3=y2-A2, m3=y2+A2, y1=m3+A2, u3=y1+A2;
  dot(TRI,v3); dot(TRI,m3); dot(TRI,y1); dot(TRI,u3);
  draw(TRI,(u3+A2)--(v3-A2),Arrow(position=1/12));
  label(TRI,"\small $s_6s_5$",u3+A2/2,E);
  // 
  point u2=y3-A3, m2=y3+A3, v2=y1+A3;
  dot(TRI,u2); dot(TRI,m2); dot(TRI,v2);
  draw(TRI,(u2-A3*3/4)--(v2+A3/2),Arrow(position=1/13));
  label(TRI,"\small $s_4s_3$",u2-A3*2/5,S);

  // 
  label(TRI,"\footnotesize $m_{34}$",m2,N);
  label(TRI,"\scriptsize $y_2=y_3$",y3,SE);
  label(TRI,"\footnotesize $m_{12}$",m1,NW);
  label(TRI,"\scriptsize $y_1=y_6$",y2,E);
  label(TRI,"\footnotesize $m_{56}$",m3,NE);
  label(TRI,"\scriptsize $y_4=y_5$",y1,ESE);
  label(TRI,"\footnotesize $u_{12}$",u1,E);
  label(TRI,"\footnotesize $u_{34}$",u2,N);
  label(TRI,"\footnotesize $u_{56}$",u3,ENE);
  label(TRI,"\footnotesize $v_{12}$",v1,ENE);
  label(TRI,"\footnotesize $v_{34}$",v2,N);
  label(TRI,"\footnotesize $v_{56}$",v3,W);

  add(TRI,(0,0));

  picture PAR;
  real abs1=-50, abs2=0, abs3=40, haut=180, h3=30, h4=haut-50;
  point x3=(abs2,h3), x4=(abs2,h4);
  real a1=36, a2=h4-h3-a1;
  point y1=(abs1,h3+a1), y2=(abs1,h3), y5=(abs3,h4), y6=(abs3,h3+a1);

  draw (PAR,(abs1,haut)--(abs1,0),Arrow(Relative(0.47)));
  label (PAR,"\footnotesize $s_2s_1$",(abs1,0.55*haut),W);
  draw (PAR,(abs2,0)--(abs2,haut),Arrow(Relative(0.92)));
  label (PAR,"\footnotesize $s_4s_3$",(abs2,0.9*haut),W);
  draw (PAR,(abs3,haut)--(abs3,0),Arrow(Relative(0.85)));
  label (PAR,"\footnotesize $s_6s_5$",(abs3,0.17*haut),E);

  dot(PAR,x3); dot(PAR,x4); dot(PAR,y1);
  dot(PAR,y2); dot(PAR,y5); dot(PAR,y6);
  label(PAR,"\footnotesize $y_4=x_4$",x4,W);
  label(PAR,"\footnotesize $x_3=y_3$",x3-(13,0),SE);
  label(PAR,"\footnotesize $y_1$",y1,W);
  label(PAR,"\footnotesize $y_2$",y2,W);
  label(PAR,"\footnotesize $y_5$",y5,E);
  label(PAR,"\footnotesize $y_6$",y6,E);
  //draw (PAR,(abs1,h3)--(abs3,h3),dashed);
  draw (PAR,y2--x3,dashed);
  draw (PAR,(abs1,h3+a1)--(abs3,h3+a1),dashed);
  draw (PAR,(abs2,h4)--(abs3,h4),dashed);

  draw(PAR,"\footnotesize $a_{12}$", (y1+(3.4,-1))--(y2+(3.4,1)),0.5*E,Arrows(TeXHead));
  draw(PAR,"\footnotesize $a_{34}$", (x4+(3.4,-1))--(x3+(3.4,1)),0.5*E,Arrows(TeXHead));
  draw(PAR,"\footnotesize $a_{56}$", (y5+(-3.4,-1))--(y6+(-3.4,1)),0.5*W,Arrows(TeXHead));

  // 
  point origPar=(240,-40);
  add(PAR,origPar);
\end{asy}
\caption{The cases (TRI), to the left, and (PAR), to the right}
\label{FiguresTRIPAR}
\end{center}
\end{figure}

Figures \ref{FiguresTRIPAR} and \ref{FigureHexT}
illustrate the cases (TRI), (PAR) and (SKH) with
some extra notation that will be used in the next section.
In particular, in every case we will consider the relevant intersection
points $y_1$, \ldots, $y_6$; they satisfy the relation
$s_{x_{i+1}}s_{x_i}=s_{y_{i+1}}s_{y_i}$ for $i=1, 3, 5$.

\begin{figure}[ht]
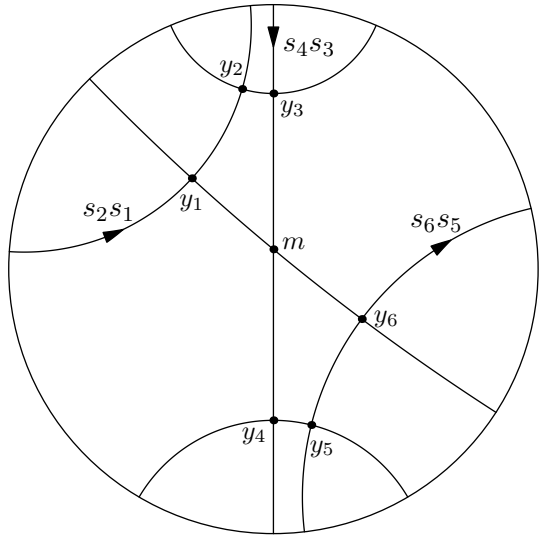

\begin{center}
\begin{asy}
  import math;
  import hyperbolic_geometry;
  real taille = 200; size(taille,taille);

  hyperbolic_point y3=hyperbolic_point(1.6,90), y4=hyperbolic_point(1.3,-90);
  hyperbolic_line droite34 = hyperbolic_line(y3,y4);
  hyperbolic_line droite23 = hyperbolic_normal(droite34,y3);
  hyperbolic_line droite45 = hyperbolic_normal(droite34,y4);
  hyperbolic_point m=hyperbolic_point(0.15,90);
  hyperbolic_point const = hyperbolic_point(8,134);
  hyperbolic_line droite16 = hyperbolic_line(m,const);
  hyperbolic_line droite12 = common_perpendicular(droite23,droite16);
  hyperbolic_line droite56 = common_perpendicular(droite45,droite16);
  hyperbolic_point y1 = intersection(droite12,droite16);
  hyperbolic_point y2 = intersection(droite12,droite23);
  hyperbolic_point y5 = intersection(droite45,droite56);
  hyperbolic_point y6 = intersection(droite16,droite56);

  draw(unitcircle);
  draw(droite34,Arrow(Relative(0.08)));
  draw(droite16); draw(droite23); draw(droite45);
  draw(reverse(droite12.to_path()),Arrow(Relative(0.3)));
  draw(droite56,Arrow(Relative(0.8)));
  label("\small $s_4s_3$",relpoint(droite34.to_path(),0.08),E);
  label("\small $s_2s_1$",relpoint(droite12.to_path(),0.7),NNW);
  label("\small $s_6s_5$",relpoint(droite56.to_path(),0.8),NNW);

  dot(m); label("\footnotesize $m$",m.get_euclidean(),ENE);
  dot(y1); label("\footnotesize $y_1$",y1.get_euclidean(),1.7*S);
  dot(y2); label("\footnotesize $y_2$",y2.get_euclidean(),1.4*NNW);
  dot(y3); label("\footnotesize $y_3$",y3.get_euclidean(),SE);
  dot(y4); label("\footnotesize $y_4$",y4.get_euclidean(),SW);
  dot(y5); label("\footnotesize $y_5$",y5.get_euclidean(),1.7*SSE);
  dot(y6); label("\footnotesize $y_6$",y6.get_euclidean(),1.3*E);
\end{asy}
\caption{The case (SKH)}
\label{FigureHexT}
\end{center}
\end{figure}

\begin{remark}\label{RmqTriangle}
  In Figure \ref{FiguresTRIPAR}
  we have drawn the triangular case as a Euclidean one for
  convenience. It is noteworthy however that
  if $\HH$ was replaced by $\R^2$ we would have an extra degree of freedom:
  the side lengths of the triangle formed by $D_{12}$, $D_{34}$ and $D_{56}$
  would only be proportional to $a_{12}$, $a_{34}$ and $a_{56}$ by a free
  positive constant. This phenomenon will have a significant role in
  Section~\ref{neighbourhood}.
\end{remark}


\section{Reduction of size}\label{SectionAlgo}

In this section we will first prove that the space $X_0$ is
contractible, then we will prove Theorem~\ref{SexMinimal};
we will end this section with miscellaneous considerations.
The strategy of both proofs consists in decreasing the relative
distances between the points $x_i$, but as the considerations of the
first proof are more differential and the second more discrete
in nature, we will use different functions to measure this size.
If $z=(x_1,\ldots,x_6)\in\Sex$, as in Section~\ref{SubsecTricho} let
$a_{i,i+1}=d(x_i,x_{i+1})$ for all $i$ with cyclic notation, and put
\[
F(z)=\prod_{i=1}^6 \cosh(a_{i,i+1}),
\quad
A(z)=a_{12}+a_{34}+a_{56} \quad \text{and} \quad
B(z)=a_{23}+a_{45}+a_{61}.
\]

\subsection{The space $X_0$ is contractible}\label{topsab1}

An elementary computation shows that the differentiable map
$(\HH)^6\rightarrow \PSL$ defined by
$(x_1,\ldots,x_6)\mapsto s_{x_6}\cdots s_{x_1}$ has
surjective differential at every point such that $x_i\neq x_j$ for
some $i,j$. Thus, the subspace $\Sex^*$
of non-singular sextuple configurations is a smooth submanifold of
$(\HH)^6$, and the map $F\colon\Sex^*\rightarrow\R$ is
differentiable.

\begin{lemma}
  The map $F$ has no critical points on $\Sex_0^*$.
\end{lemma}

\begin{proof}
  By contradiction, suppose $z=(x_1,\ldots,x_6)\in\Sex_0^*$
  is a critical point for $F$.

  If $z$ is an aligned configuration, we can push the points
  together by rescaling simultaneously $a_{i,i+1}$ for all $i$
  by the same factor, thus decreasing $F$ at first order, a
  contradiction.

  Suppose now that $x_i\neq x_{i+1}$ for all $i$.
  Consider the line $D_{12}$, oriented from $x_1$ to $x_2$ and
  consider a flow $\phi_{12}^t$ on a neighbourhood
  of $z$ in $\Sex$
  replacing $x_1$ and $x_2$ by their image by the translation of
  length $t$ along $D_{12}$. Let $x_6'$ and $x_3'$ be the orthogonal
  projections of $x_6$ and $x_3$ on $D_{12}$.
  By the hyperbolic Pythagorean theorem, we have
  $\cosh(a_{61})=\cosh(d(x_6,x'_6))\cosh(d(x'_6,x_1))$ and
  $\cosh(a_{23})=\cosh(d(x_3,x'_3))\cosh(d(x'_3,x_2))$,
  hence
  \[
  F(\phi_{12}^t(z))=C\cosh(d(x_6',x_1)\pm t)\cosh(d(x_2,x_3')\pm t),
  \]
  where
  $C=\cosh(a_{12})\cosh(a_{34})\cosh(a_{45})\cosh(a_{56})
  \cosh(d(x_6,x_6')\cosh(x_3,x_3')$
  does not depend on $t$, and where the signs before $t$ depend on the
  order of the points $x_1$, $x_2$, $x_6'$ and $x_3'$ on the line $D_{12}$.
  In either case, by deriving the above expression, we check that the
  criticality of $F$ at $z$ implies that the midpoints of the segments
  $[x_1,x_2]$ and $[x_3',x_6']$ coincide; similar conclusions hold if we
  cyclically permute the~$x_i$'s.

  Put $x=d(x_1,y_1)$, $y=d(x_3,y_3)$ and $z=d(x_5,y_5)$. Note that we
  also have $x=d(x_2,y_2)$, $y=d(x_4,y_4)$ and $z=d(x_6,y_6)$. So the
  midpoint of $[x_1,x_2]$ is at distance $x$ from $m_{12}$,
  the midpoint of $[y_1,y_2]$.

  If we are in case (TRI) or (SKH), the projections of $y_3$ and
  $y_6$ on $D_{12}$ are $y_2$ and $y_1$.
  Since orthogonal projections decrease
  distances, the midpoint of $[x_3',x_6']$ is at distance
  less than
  $\frac{y+z}{2}$ of $m_{12}$, with equality if and only if $y=z=0$.
  Thus we have the inequality $2x\leq y+z$, and its three cyclic
  companions. It follows that $x=y=z=0$.
  In case (TRI) this contradicts that $x_2\neq x_3$. In case (SKH)
  we can shorten $a_{34}$ while fixing $a_{12}$ and $a_{56}$, by
  classical hyperbolic formulas (see eg \cite{Buser}, page 454)
  this shortens $a_{23}$, $a_{45}$ and $a_{61}$, contradicting that
  $z$ is a critical point of~$F$.

  If we are in case (PAR), we again contradict that $z$ is a critical
  point of $F$ by pushing simultaneously all the points $x_i$ towards
  the common point at infinity of $D_{12}$, $D_{34}$ and $D_{56}$:
  this leaves $a_{12}$, $a_{34}$ and $a_{56}$ invariant while decreasing
  the other three distances.

  Finally, suppose, say, that $x_1=x_2$.
  Then $x_3$, \ldots, $x_6$ lie in a same line $D$, well-defined since
  the configuration is not aligned.
  Let $x$ be the
  orthogonal projection of $x_1$ on $D$. Then we may push simultaneously
  the points $x_3$, \ldots, $x_6$ towards $x$ along $D$, reducing all the
  distances $a_{i,i+1}$ at the first order, thus $z$ cannot be a
  critical point of~$F$.
\end{proof}
Every element of $\Sex^*$ has a closed $\PSL$-orbit;
it follows that $X^*$ is smooth: the only singularity of $X$
is the singular configuration. It is noteworthy that although the aligned,
non-singular configurations are smooth points of $X^*$, their
images in $X_0(\Gamma)$ are non-smooth points of the character variety
of the genus two surface group.

Now, the map $F\colon\Sex_0^*\rightarrow(1,+\infty)$ is $\PSL$-invariant,
hence it descends to a map
$f\colon X_0^*\rightarrow(1,+\infty)$ which still has surjective
differential at every point. Obviously $f$ is onto; it is also
proper.
Therefore, $X_0^*$ is diffeomorphic
to $(1,+\infty)\times f^{-1}(\{r\})$ for any $r>1$, and $X_0$ is homeomorphic to a cone
over $f^{-1}(\{r\})$. This proves that $X_0$ is contractible,
and homeomorphic to a neighbourhood of the singular configuration.
This neighbourhood will be described precisely in
Section~\ref{neighbourhood}.


\subsection{An effective method}\label{SectEffective}

\subsubsection{The set $U$}

In this section we consider the subset $U$ of non-pinched
isometry classes of sextuple configurations.
Before proving Theorem \ref{SexMinimal}, let us prove the following
statement, which will be needed in Section~\ref{SectionDyn}.
\begin{observation}\label{UConnexe}
  The set $U$ is connected, is dense and has full measure in $X_0$.
\end{observation}
\begin{proof}
  Consider the following set
  $P=\{(x_1,\ldots,x_6)\in\Sex\,|\,x_1=x_2, x_3=x_4, x_5=x_6, x_2\neq x_3\}$
  of pinched, non-singular configurations. It is a submanifold of dimension
  $6$ of $\Sex^*$, and it descends to a submanifold of dimension $3$,
  hence of codimension $3$, of $X_0$. Now $U$ is the complement of the
  orbit of this codimension $3$ submanifold, under the countable group
  $\mcg(S_o)$; the statement of the observation follows.
\end{proof}
Now we turn to the proof of Theorem \ref{SexMinimal}, restated as follows.
\begin{theorem}
  For all $[z]\in U$, the adherence $\overline{\mcg(S_o)\cdot[z]}$
  contains the isometry class of singular configurations.
\end{theorem}

\subsubsection{The operations}\label{ParAlg}
To prove this statement we show how to construct an effective
``geometric algorithm'' (we will comment later on this terminology)
which reduces the size, measured by the functions
$A$ and $B$ introduced above, of sextuple configurations.
Suppose $z=(x_1,\ldots,x_6)$ is a non-aligned sextuple satisfying
$x_i\neq x_{i+1}$ for $i=1,3,5$. Consider the following operations on $z$,
depending on
the trichotomy of Lemma~\ref{tricho}.

\begin{itemize}
\item[(Rot)]Make a cyclic permutation, to obtain a sextuple
  $z'$ with $A(z')=B(z)$ and $B(z')=A(z)$.
\item[(Tri)] If we are in case (TRI), apply (Rot) an even number of times,
  so that $a_{12}\geq a_{34}$ and $a_{12}\geq a_{56}$. Then perform the
  leapfrog moves $L_1$ to a power minimizing the distance $d(x_1,y_1)$;
  apply similarly the moves $L_3$ and $L_5$. Then, if this further
  decreases $B$, apply again $L_1$ or its inverse, once. Then apply (Rot).
\item[(Par)] If we are in case (PAR), apply leapfrog moves $L_3$ to
  push $x_3$ and $x_4$ towards the common point at $\partial\HH$ of
  $D_{12}$, $D_{34}$ and $D_{56}$, until the distances along horocycles,
  from $x_3$ or $x_4$, to $D_{12}$ and $D_{56}$, become less than
  $\frac{A(z)}{12}$.
  Then consider the points $y_1$, $y_2$, $y_5$ and $y_6$ as in Figure
  \ref{FiguresTRIPAR}, and apply powers of $L_1$ and $L_5$ to minimize
  the distances $d(x_1,y_1)$ and $d(x_5,y_5)$. Then apply (Rot).
\item[(Skh0)] If we are in case (SKH), first apply (Rot) an even number of
  times so that $a_{34}>a_{12}+a_{56}$, as in Figure \ref{FigureHexT}.
  Apply powers of $L_1$ and $L_5$ to minimize
  $d(x_1,y_1)$ and $d(x_5,y_5)$. Apply
  a power of $L_3$ to put $x_3$ in the segment $[y_3,m]$ or $x_4$
  in the segment $[m,y_4]$.
\item[(Skh1)] If in case (SKH), apply (Skh0), and then apply
  the ``half-twist move'' $(L_1L_2L_1)^2$.
\end{itemize}
The effect of these operations is precised in the following lemmas.
\begin{lemma}\label{LemmeTRIPAR}
  Suppose $z$ is in case (TRI) or (PAR), and let $z'$ be the sextuple
  resulting from applying the relevant operation, (Tri) or (Par), to $z$. Then
  $\frac{24}{23}A(z')\leq B(z')=A(z)$.
\end{lemma}
In the next lemmas, we write $b_{i,i+1}=d(y_i,y_{i+1})$ for $i=2,4,6$, in the
case (SKH), so that the right-angled skew hexagon has edge lengths
$a_{i,i+1}$ with $i=1,3,5$ and $b_{i,i+1}$ with $i=2,4,6$.
\begin{lemma}\label{LemmeSKH1}
  Suppose $z$ is in case (SKH), and let $z'$ be the sextuple resulting from
  applying the operation (Skh1) to $z$. Then $A(z')<A(z)$. More precisely,
  if $a_{34}>a_{12}+a_{56}$, then the operation (Skh1) leaves
  $a_{12}$ and $a_{56}$ invariant and decreases $\cosh(a_{34})$ by
  a quantity larger than
  $2\frac{\sinh^2(\min(a_{12},a_{56}))}{\cosh^2(a_{34})}$.
  Moreover, provided $A(z)$ is small enough, we also have
  $B(z')\leq B(z)+4A(z)-\min(1,b_{23},b_{45})$.
\end{lemma}
The operations described above do not suffice yet to prove
Theorem~\ref{SexMinimal}; the next lemma introduces an additional operation.
\begin{lemma}[Operation (Skh2)]\label{LemmeSKH2}
  Provided $A(z)$ and $\min_{i=2,4,6}(b_{i,i+1})$ are small enough,
  after applying (Skh0) the isometry
  $s_{x_3'}s_{x_2'}s_{x_1'}$ is a rotation of angle close
  (but not equal) to $\pi$.
  Then there exists $N\geq 0$ such that
  $(s_{x_3'}s_{x_2'}s_{x_1'})^N$ is a rotation of angle close to
  $\pm \frac{\pi}{2}$ and such that the move
  $(L_1L_2L_1)^{2N}$ results in decreasing $A$, and
  such that, if the resulting configuration $z'$ is still in case
  (SKH),
  then after applying again (Skh0) we have $B(z')\leq B(z)-1$.
\end{lemma}

\subsubsection{Proof of Theorem \ref{SexMinimal}}
Let $z=(x_1,\ldots,x_6)\in\Sex$ such that $[z]\in U$, and let
$\varepsilon>0$, small enough to apply Lemmas~\ref{LemmeSKH1}
and~\ref{LemmeSKH2}. We want to prove that there exists
$z'\in B_6(S^2)\cdot z$ such that $A(z')+B(z') \leq 2\varepsilon$.

Let us suppose that the configuration is not aligned; we postpone
the aligned case to the end of the proof.
The case in which $x_i=x_{i+1}$ for some $i$ is quite straightforward
and we will deal with it later.
Thus, let us suppose
now that our configuration $z$, as well as all the configurations we
deal with in the following process, satisfy the condition
$x_i\neq x_{i+1}$ for all~$i$.

As a first step, let us apply the operations (Tri) or (Par) or (Skh1),
depending on whether $z$ is in case (TRI), (PAR) or (SKH) of Lemma~\ref{tricho},
and iterate this procedure, until $A\leq \varepsilon$.
This first process stops in finite time. Indeed, in cases (TRI) and (PAR),
$A$ drops by a factor of at least $\frac{24}{23}$. The worst thing that
could happen is to encounter only the case (SKH) after some iteration.
Suppose it is the case. By Lemma~\ref{LemmeSKH1}, at each iteration,
the three quantities $a_{12}$, $a_{34}$ and $a_{56}$ all decrease, hence
the biggest of them stays smaller than the starting quantity $A(z)$.
Hence, by Lemma~\ref{LemmeSKH1}, at each iteration the biggest
of $\cosh(a_{12})$, $\cosh(a_{34})$ and $\cosh(a_{56})$ decreases,
by an additive amount depending only on the smallest.
Hence, should the process not stop in finite time,
the smallest of $a_{12}$, $a_{34}$ and $a_{56}$ would have to converge to $0$.
But each iteration changes only the biggest of $a_{12}$, $a_{34}$ and $a_{56}$.
Hence, these three quantities converge to $0$, and this first stage of
iterations does stop in finite time.

Now we have $A(z')\leq\varepsilon$. Of course, under this condition,
if we encounter again the case (TRI) or (PAR) then we are done, by
Lemma~\ref{LemmeTRIPAR}. Suppose we do not.
Proceed as in the first step until we have
$\min(b_{23},b_{45},b_{61})\leq \varepsilon$ in case (SKH).
This will happen in finite time, by the second assertion of
Lemma~\ref{LemmeSKH1}. Note that the operation
(Skh0) leaves $A$ invariant, and after this operation we have
$B\leq b_{23}+b_{45}+b_{61}+2A$.
So as a last step, we proceed by iterating either the operation
(Skh1), or ((Skh2) followed by (Skh0)), depending on which decreases
$B$ the most. The conclusions of
Lemmas~\ref{LemmeSKH1} and~\ref{LemmeSKH2} now imply that $B$ converges
to $0$, hence is lower than $\varepsilon$ after finitely many
iterations.

Let us deal now
with the case when $x_i=x_{i+1}$ for some~$i$. Up to applying (Rot),
suppose that $x_5=x_6$. Now $s_{x_4}s_{x_4}s_{x_2}s_{x_1}=1$, so
$x_1$, $x_2$, $x_3$ and $x_4$ have to lie on a same line. Since
$[z]\in U$, these four points need to be pairwise distinct,
otherwise we could easily produce leapfrog moves leading to a
configuration where $x_i=x_{i+1}$ for $i=1,3,5$.
Denote by $\Delta$ be the line containing $x_1,\ldots,x_4$ and orient this
line. Up to applying powers of $L_3$ we may suppose that
$x_1$ is to the left of $x_2$, $x_3$ and $x_4$ on $\Delta$; this
implies that $x_3$ is at the right side of $x_1$, $x_2$, $x_4$.
Now put $\delta_1=d(x_1,x_2)$ and $\delta_2=d(x_1,x_4)$.
If $\delta_2>\delta_1$ apply the move $L_3$, this changes
$(\delta_1,\delta_2)$ into $(\delta_1,\delta_2-\delta_1)$. If
$\delta_1<\delta_2$ apply the move $L_2^{-1}$; this changes
$(\delta_1,\delta_2)$ into $(\delta_1-\delta_2,\delta_2)$.
We recognise Euclid's algorithm. It follows from the condition
$[z]\in U$ that $\delta_1$ and $\delta_2$ have an irrational ratio,
so this process pushes the points $x_1$, $x_2$, $x_3$ and $x_4$ close
together. By applying then powers of $L_1$ and $L_3$ we may
now push these four points as close as we want to the projection of
$x_5=x_6$ on $\Delta$. If $x_5=x_6\in\Delta$ then we are done.
Otherwise, according to the construction of Section~\ref{SubsecTroiSym},
the isometry $s_{x_6}s_{x_1}s_{x_2}$ is a rotation, of center
as close as we want from $x_5$, and of angle close (but distinct) to $\pi$,
hence it has an $N$-th power with angle close to $\frac{\pi}{2}$.
Then apply the iterated half-twist, $(L_6L_1L_6)^{2N}$.
This results in a configuration with $A$ as small as we want
(hence $A\leq\varepsilon$), furthermore in this new configuration the
lines $(x_1,x_2)$ and $(x_3,x_4)$ now cross each other, hence this
configuration is in case (TRI). Thus, it remains to apply the
operation (Tri) in order to get $A\leq\varepsilon$ and $B\leq\varepsilon$.

The case of aligned configurations, finally, can be treated by
mixing the strategy of case (PAR) and that of the degenerate case
above, depending on whether $x_i=x_{i+1}$ for some~$i$.
\qed

\subsubsection{Proofs of the lemmas}
\begin{proof}[Proof of Lemma \ref{LemmeTRIPAR}]
  The leapfrog moves $L_1$, $L_3$ and $L_5$ obviously do not
  change the distances $a_{12}$, $a_{34}$ and $a_{56}$ summing up to $A(z)$.
  We need to prove that the moves as in Operations (Tri) and (Par),
  except the last rotation, lead to a configuration $z'$ with
  $B(z')<\frac{23}{24}A(z)$.
  Here we use the notation $z'$ even though we may not
  have reached yet the configuration as in the statement of the lemma; the
  notation $z'$ is subject to change in the course of the proof, accordingly
  to the appropriate moves. We make this abuse of notation here, and in the
  two subsequent proofs.

  Suppose first that $z$ is in the case (TRI). After applying the leapfrog
  moves minimizing the distances $d(x_i',y_i)$, we have
  $x_i'\in[u_{i,i+1},m_{i,i+1}]$ for $i=1,3,5$ and
  $x_i'\in[m_{i-1,i},v_{i-1,i}]$ for $i=2,4,6$, where the points
  $m_{i,i+1}$, $u_{i,i+1}$ and $v_{i,i+1}$, $i=1,3,5$ are as in
  Figure~\ref{FiguresTRIPAR}.
  Thus,
  \begin{align*}
  d(x_1',y_1)\leq\frac{a_{12}}{2},\quad d(x_2',y_2)\leq\frac{a_{12}}{2}, \quad
  d(x_3',y_3)\leq\frac{a_{34}}{2}, \\
  d(x_4',y_4)\leq\frac{a_{34}}{2},\quad d(x_5',y_5)\leq\frac{a_{56}}{2}, \quad
  d(x_6',y_6)\leq\frac{a_{56}}{2}.
  \end{align*}
  As $y_1=y_6$, $y_2=y_3$ and $y_4=y_5$ this already gives $B(z')\leq A(z)$,
  by triangular inequalities.

  If $B(z')\geq\frac{23}{24}A(z)$, then the distances of each $x_i'$ to
  the closest end of its segment sum up to less than $\frac{A(z)}{24}$.
  Suppose (without loss of generality) that $a_{56}\leq a_{34}\leq a_{12}$, so that
  $a_{56}\leq\frac{A(z)}{3}$.
  The symmetry
  around $y_2$ and the CAT(0) inequality (and the
  intercept theorem) imply
  $d(v_{12},u_{34})=d(m_{12},m_{34})\leq\frac{a_{56}}{2}$.
  If $d(x_2',m_{12})+d(x_3',m_{34})\leq\frac{A(z)}{24}$, or
  $d(x_2',v_{12})+d(x_3',u_{34})\leq\frac{A(z)}{24}$, then by
  triangle inequalities we get $B(z')\leq \frac{21}{24}A(z)$, a
  contradiction. Thus $d(x_2',v_{12})+d(x_3',m_{34})\leq\frac{A(z)}{24}$ or
  $d(x_2',m_{12})+d(x_3',u_{34})\leq\frac{A(z)}{24}$. After an extra
  leapfrog move $L_1^{\pm 1}$ we have again
  $d(x_2',x_3')\leq\frac{a_{56}}{2}+\frac{A(z)}{24}$, and now we
  only have $d(x_1',y_1)\leq\frac{a_{12}}{2}+\frac{A(z)}{24}$. Triangle
  inequalities this time give $B(z')\leq \frac{22}{24}A(z)$; this
  settles the triangle case.

  The parabolic case is much simpler: the inequalities
  $d(y_2,x_3')\leq\frac{A(z)}{12}$, $d(x_4',y_5)\leq\frac{A(z)}{12}$,
  $d(y_1,y_6)\leq\frac{A(z)}{6}$, $d(x_i',y_i)\leq\frac{a_{12}}{2}$
  for $i=1,2$ and $d(x_i',y_i)\leq\frac{a_{56}}{2}$ for $i=5,6$
  directly yield, by triangle inequalities,
  $B(z')\leq\frac{5}{6}A(z)$.
\end{proof}

\begin{proof}[Proof of Lemma \ref{LemmeSKH1}]
  Let $z=(x_1,\ldots,x_5)\in\Sex$ be in the (SKH) configuration.
  Perform the moves according to the operation (Skh1), except the last
  ``half-twist''.
  Without loss of generality, suppose that now $x_3'\in[y_3,m]$.
  The moves we made so far do not change the value of $A$, and
  clearly, the half-twist then does not change the value of $a_{12}$ or
  $a_{56}$. In order to prove the first statement of Lemma~\ref{LemmeSKH1}
  we study the effect of this half-twist on $a_{34}$.

  This half-twist amounts to replace $x_1'$, $x_2'$ and $x_3'$
  by their image by the isometry $s_{x_3'}s_{x_2'}s_{x_1'}$.
  Recall from Section~\ref{SubsecTroiSym} that this isometry is the
  product $r_{\ell_2}r_{\ell_1}$ of the reflections by the lines $\ell_1$
  and $\ell_2$, as constructed in Figure~\ref{FigureHexTDeux}.
  \begin{figure}[htb]
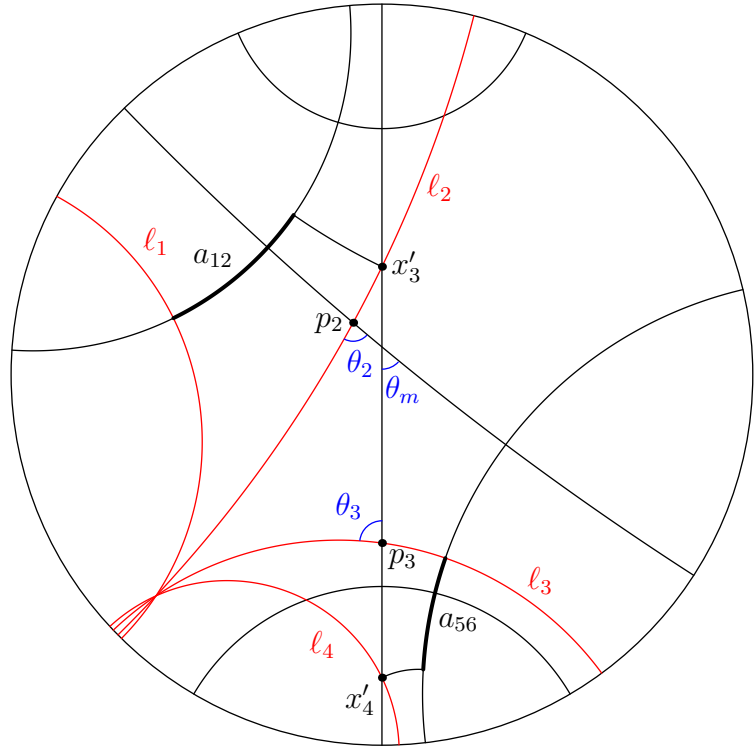

  \begin{center}
  \begin{asy}
  import math;
  import hyperbolic_geometry;
  import geometry;
  real taille = 280;
  size(taille,taille);

  hyperbolic_point y3=hyperbolic_point(1.6,90);
  hyperbolic_point y4=hyperbolic_point(1.3,-90);
  hyperbolic_line droite34 = hyperbolic_line(y3,y4);
  hyperbolic_line droite23 = hyperbolic_normal(droite34,y3);
  hyperbolic_line droite45 = hyperbolic_normal(droite34,y4);
  hyperbolic_point m=hyperbolic_point(0.15,90);
  hyperbolic_point const = hyperbolic_point(8,134);
  hyperbolic_line droite16 = hyperbolic_line(m,const);
  hyperbolic_line droite12 = common_perpendicular(droite23,droite16);
  hyperbolic_line droite56 = common_perpendicular(droite45,droite16);
  hyperbolic_point y1 = intersection(droite12,droite16);
  hyperbolic_point y2 = intersection(droite12,droite23);
  hyperbolic_point y5 = intersection(droite45,droite56);
  hyperbolic_point y6 = intersection(droite16,droite56);

  hyperbolic_point x3 = hyperbolic_point(0.6,90);
  hyperbolic_line segment1 = hyperbolic_normal(droite12,x3);
  hyperbolic_point cc2 = intersection(droite12,segment1);
  hyperbolic_line ell2 = hyperbolic_normal(segment1,x3);

  hyperbolic_point x4 = hyperbolic_point(2.3,-90);
  hyperbolic_line segment2 = hyperbolic_normal(droite56,x4);
  hyperbolic_point cc5 = intersection(droite56,segment2);
  hyperbolic_line ell4 = hyperbolic_normal(segment2,x4);

  hyperbolic_point COMM = intersection(ell2,ell4);
  point CommEucl = COMM.get_euclidean();
  hyperbolic_line ell1 = hyperbolic_normal(droite12,COMM);
  hyperbolic_line ell3 = hyperbolic_normal(droite56,COMM);

  hyperbolic_point p3 = intersection(ell3,droite34);
  hyperbolic_point p2 = intersection(ell2,droite16);

  hyperbolic_point ccc1 = intersection(ell1,droite12);
  hyperbolic_point ccc6 = intersection(ell3,droite56);

  draw(unitcircle);
  draw(droite34); draw(droite16); draw(droite23);
  draw(droite45); draw(droite12); draw(droite56);
  label("$\ell_1$",relpoint(ell1.to_path(),0.2),NE,red);
  label("$\ell_2$",relpoint(ell2.to_path(),0.25),E,red);
  label("$\ell_3$",relpoint(ell3.to_path(),0.82),NNE,red);
  label("$\ell_4$",relpoint(ell4.to_path(),0.65),SW,red);

  draw(ell2,red); draw(ell4,red); draw(ell1,red); draw(ell3,red);

  dot(p3); label("$p_3$",p3.get_euclidean(),SE);
  dot(p2); label("$p_2$",p2.get_euclidean(),W);

  dot(x3); label("$x_3'$",x3.get_euclidean(),E);
  dot(x4); label("$x_4'$",x4.get_euclidean(),SW);
  draw(hyperbolic_segment(cc2,x3)); draw(hyperbolic_segment(cc5,x4));

  draw(hyperbolic_segment(ccc1,cc2),black+1.5pt);
  draw(hyperbolic_segment(ccc6,cc5),black+1.5pt);

  label("$a_{12}$",relpoint(droite12.to_path(),0.55),NW);
  label("$a_{56}$",relpoint(droite56.to_path(),0.2),E);

  draw("$\theta_3$",arc(p3.get_euclidean(),0.06,90,173),blue);
  draw("$\theta_2$",arc(p2.get_euclidean(),0.05,-118,-43),blue);
  draw("$\theta_m$",arc(m.get_euclidean(),0.06,-89,-43),blue);
  \end{asy}
  \caption{The operation (Skh1)}
  \label{FigureHexTDeux}
  \end{center}
  \end{figure}
  Construct, similarly, the lines $\ell_3$ and $\ell_4$ such that
  $s_{x_6'}s_{x_5'}s_{x_4'}=r_{\ell_3}r_{\ell_4}$.
  As $z\in\Sex$, we have $r_{\ell_1}r_{\ell_2}=r_{\ell_3}r_{\ell_4}$.
  In particular these four lines either intersect in $\HH$, or in
  $\partial\HH$, or are all perpendicular to a common line, depending
  on whether $s_{x_3'}s_{x_2'}s_{x_1'}$ is elliptic, parabolic or
  hyperbolic: this is what we observe in Figure~\ref{FigureHexTDeux},
  but our reasoning will not depend on this trichotomy.

  We want to compare $a_{34}$ with the new distance,
  $d(r_{\ell_2}r_{\ell_1}\cdot x_3',x_4')$. This last distance is
  equal to $d(x_3',r_{\ell_3}r_{\ell_4}\cdot x_4')=d(x_3',r_{\ell_3}x_4')$.
  Let $p_3\in\HH$ and $\theta_3\in\R$ be as in Figure~\ref{FigureHexTDeux},
  the intersection point of $\ell_3$ and $D_{34}$ and their angle.
  Set $u=d(x_3',p_3)$ and $v=d(p_3,x_4')$.
  Then $\cosh(a_{34})=\cosh u\cosh v+\sinh u\sinh v$,
  whereas
  $\cosh(d(x_3',r_{\ell_3}x_4'))=\cosh u\cosh v+\cos(2\theta_3)\sinh u\sinh v$,
  hence the decrement in $\cosh(a_{34})$ is equal to
  $2\sin^2\theta_3\sinh u\sinh v$.
  We claim that $u\geq \min(a_{12},a_{56})$, $v\geq a_{56}$, and
  $\sin\theta_3\geq\frac{1}{\cosh(a_{34})}$; these claims imply the first part
  of Lemma \ref{LemmeSKH1}.

  The inequality $v\geq a_{56}$ is obvious.
  Also, we have $d(p_3,y_4)\leq a_{34}$. Thus, the lines
  $\ell_3$, $D_{34}$ and $(y_4,y_5)$ form the same picture as in
  Figure~\ref{FigureHexSimple}, where the angle and the distance involved
  are $\theta_3$ and $d(p_3,y_4)$. Indeed, $\ell_3$ and $(y_4,y_5)$ being
  both perpendicular to the same line $D_{56}$, do not cross each other.
  The claimed inequality follows.
  \begin{figure}[htb]
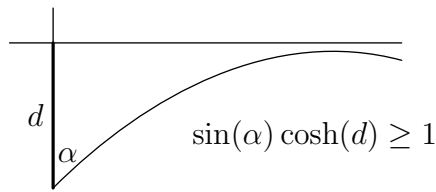

  \begin{center}
  \begin{asy}
  import geometry;
  unitsize(1.1pt);

  draw((-15,0)--(120,0)); draw((0,12)--(0,-50));
  draw((0,0)--(0,-50),black+1.2pt);
  draw((0,-50){dir(45)}..(120,-6){dir(-15)});
  label("$d$",(0,-25),W); label("$\alpha$",(0,-50),3*dir(68));
  label("$\sin(\alpha)\cosh(d)\geq 1$",(90,-32));
  \end{asy}
  \end{center}
  \caption{A classical inequality}
  \label{FigureHexSimple}
  \end{figure}
  The remaining inequality, $u=d(x_3,p_3)\geq\min(a_{12},a_{56})$, is
  obtained by a variation argument.
  If $x_3=y_3$ then $p_3=m$, so $u\geq a_{12}$.
  Starting there, as $x_3$ moves down, $x_4$ moves down at the same speed,
  and $v$ decreases until it reaches a minimum when $y_4$ is the middle
  of the segment $[p_3,x_4]$. It then increases in a symmetric way as $x_3$
  moves further down, until it reaches its initial value as $p_3=y_4$.
  Along the way we always have $u\geq a_{12}$. If, in this process, $x_3'$
  has already reached $m$, then we are done. Otherwise,
  after $p_3=y_4$
  we have $d(x_3,p_3)\geq d(m,y_4)\geq a_{56}$.
  This proves the first part of Lemma~\ref{LemmeSKH1}.

  Before applying (Skh0) we had $B(z)\geq b_{23}+b_{45}+b_{61}$.
  Indeed, $y_2$ and $y_3$ realize the smallest distance between a point
  in $D_{12}$ and a point in $D_{34}$, and so on. Let us write
  $(x_1,\ldots,x_6)$ the points after applying (Skh0), and
  $(x_1',\ldots,x_6')$ the points after applying $(L_1L_2L_1)^2$.
  Concerning $B$, this last move does not affect $d(x_2,x_3)$
  or $d(x_4,x_5)$; it changes only $d(x_6,x_1)$, by applying
  $r_{\ell_2}r_{\ell_1}$ to $x_1$. We have $d(x_6,y_6)\leq\frac{a_{56}}{2}$,
  and $d(r_{\ell_1}x_1,y_2)\leq 4a_{12}$. So
  $d(x_6',x_1')\leq\frac{a_{56}}{2}+4a_{12}+d(r_{\ell_2}y_2,y_6)$.
  Also, $d(x_2',x_3')\leq b_{23}+\frac{a_{12}}{2}+a_{34}$ and
  $d(x_4',x_5')\leq b_{45}+a_{34}+\frac{a_{56}}{2}$, so
  $B(z')\leq d(r_{\ell_2}y_2,y_6)+4A(z)+b_{23}+b_{45}$.
  Now, we obviously have $d(y_2,p_2)\geq b_{23}$, $d(p_2,y_6)\geq b_{45}$
  and $\frac{\pi}{2}\geq\theta_2\geq\theta_m$, with the notation of
  Figure~\ref{FigureHexTDeux}, and $\theta_m$ can be forced to be as close
  to $\frac{\pi}{2}$ as we want provided $A(z)$ is small; this is made
  precise by Figure~\ref{FigureHexSimple}.
  For $i=2,6$ let $q_i\in[p_2,y_i]$ the point at distance
  $\delta=\min(1,b_{23},b_{45})$ of $p_2$. Provided $\pi-2\theta_2$
  is smaller than the angle of an equilateral triangle of side $1$ in
  $\HH$, we have (eg, by CAT(0) inequality) $d(r_{\ell_2}q_2,q_6)\leq \delta$,
  hence $d(r_{\ell_2}y_2,y_6)\leq b_{61}-\delta$ by triangle inequalities;
  thus finally $B(z')\leq b_{23}+b_{45}+b_{61}+4A(z)-\min(1,b_{23},b_{45})$.
\end{proof}

\begin{proof}[Proof of Lemma \ref{LemmeSKH2}]
  If $A(z)$ is small, and, say, $b_{23}$ is small, then after doing
  (Skh0), as of the construction of Section~\ref{SubsecTroiSym},
  $s_{x_3}s_{x_2}s_{x_1}$ is a rotation of angle close to $\pi$,
  with center very close to the segment $[x_3,x_4]$. For some
  $N$, its $N$th power has angle close to $\frac{\pi}{2}$ and
  $d((x_{x_3}s_{x_2}s_{x_1})x_3,x_4)\leq d(x_3,x_4)$.
  Thus the operation (Skh2) does not increase $A$. Also, the lines
  $(x_1',x_2')$ and $(x_5',x_6')$ now cross each other
  (hence we end in case (TRI)), unless the distance
  $d(y_2,y_6)$ was very big. In that case, the distance between
  $D_{12}$ and $D_{56}$ decreases significantly by the operation
  (Skh2), and we see easily that $B(z')\leq B(z)-1$ after doing
  (Skh0) once again; this inequality is actually extremely far from
  sharp.
\end{proof}

\subsection{Side remarks}

We end this section with some remarks on which we chose not to
expand too much the exposition in this article.

First, as we said in Paragraph~\ref{ParAlg}, Theorem~\ref{SexMinimal}
is proved by iterating an explicit geometric procedure, by compass and
straightedge construction. We cannot properly speak of an
algorithm only because the data of six points in the plane is not a
finite information in terms of a finite alphabet:
it seems preferable to speak of a ``geometric algorithm'', or
``real number algorithm''.
In \cite{GilmanMaskit} and \cite{Gilman}, J.~Gilman and B.~Maskit gave
such an algorithm to decide whether a given
non-elementary
representation of the free
group of rank $2$ in $\PSL$ is discrete. Using this,
the ``algorithm'' given in the paragraph above decides whether a
non-elementary
representation of the group $\Gamma_o$ is discrete. Indeed, it follows
from Margulis' lemma, and Theorem~\ref{SexMinimal}, that a
representation in $\Hom'(\Gamma_o,\PSL)$
can be discrete only if it is pinched,
in which case we can run the Gilman-Maskit algorithm on the two generators
$s_{x_3}s_{x_1}$ and $s_{x_5}s_{x_1}$, which generate an index-two
(or index-one, accidentally)
subgroup of the image of our representation
(the case of non-elementary representations of $\Gamma_o$
which kill no $c_i$ is easier: all of these representations
are discrete).
In this regard, we can certainly replace $U$ by the set of
non-discrete, non-elementary representations in the statement of
Theorem~\ref{SexMinimal}. The following statement, more general and in the
spirit of \cite{PreviteXia} seems reasonable:
\begin{conj}
  For every non-elementary, non-discrete representation $\rho\colon\Gamma_o\rightarrow\PSL$,
  the orbit $B_6(S^2)\cdot[\rho]$ is dense in $X_0$.
\end{conj}
If we replace $\PSL$ by
$\PSLC$, we have a natural identification between the
space of non-elementary morphisms of $\Gamma_o$ in $\PSLC$ which
kill no $c_i$, and all the non-elementary morphisms of $\Gamma$ in
$\PSLC$ of Stiefel-Whitney class $0$; this follows
for instance from the arguments of \cite{Mod2}, Section 3.
These representations of $\Gamma_o$ send each generator $c_i$ to
a rotation of angle $\pi$ around some line in $\HHtrois$.
The representations treated in this article correspond to six lines
orthogonal to a common plane. The representations in $X_0^-(\Gamma)$
correspond to configurations of six lines in a plane; the remaining
real characters of representations in $SO(3)$ correspond to six lines
through a common point.
The quickest proof of the connectedness of $X_0^-(\Gamma)$, to our mind,
is through this correspondence, by writing, in that setting, the analog
of Lemma~\ref{tricho}; this is all elementary and we leave the details
to the reader. We can extend the methods of Theorem~\ref{SexMinimal}
to these representations in $X_0^-(\Gamma)$, leading to a much nicer
proof of Theorem~1.4 of \cite{Mod2} in the case of Euler class $0$; we
chose not to elaborate on this point in this article.
It seems more interesting, but also quite
challenging, to find
a subset of representations in $\PSLC$, of positive measure,
on which Theorem~\ref{SexMinimal} could extend.


\section{Neighbourhood of the singular representation}\label{neighbourhood}

We saw in Theorem \ref{SexMinimal} that the orbit of
almost every representation in $X_0(\Gamma_o)$ accumulates
to the singular representation. The classification of ergodic
components of the character varieties therefore reduces to
a careful study of a neighbourhood of the singular
representation in $X_0(\Gamma_o)$. This neighbourhood
turns out to have a very rich structure; we devote this section 
to studying it. This will provide
all the material needed
to prove the ergodicity statements in Section~\ref{SectionDyn}.

\subsection{A Euclidean model}\label{euclidean}
Suppose that $z^n=(x_1^n,\ldots,x_6^n)$ is a sequence of
sextuples converging to the singular configuration. Then, up
to extraction and renormalization by a scalar, one can suppose
that the family $z^n$ converges in Gromov-Hausdorff sense to
a configuration $(p_1,\ldots,p_6)$ in the Euclidean plane
$\mathcal{E}$. Consider the set of all limiting configurations up
to affine isometry respecting the orientation: the condition
$s_6^n\cdots s_1^n=1$ implies the same condition for the
Euclidean $\pi$-rotations over the $p_i's$. This is equivalent to
the condition $\sum (-1)^i p_i=0$.
Almost every
Euclidean configuration is triangular,
and
these limits of hyperbolic configurations also have to satisfy an
extra condition reminiscent from Remark \ref{RmqTriangle}. We
leave it to the reader as a pleasant exercise in plane Euclidean
geometry that this condition is equivalent to the equality between
signed areas as appearing in the following definition:
\[X_{\mathcal E}=\{(p_1,\ldots,p_6)\in \mathcal{E}^6/ \sum (-1)^i
p_i=0, \operatorname{Area}(p_1,p_2,p_3)+ 
\operatorname{Area}(p_4,p_5,p_6)=0\}/\rm{Isom}^+(\mathcal{E}).\]
This space is the quotient by a circle action of a quadratic cone
inside some vector space $V$.
More precisely,
consider the space
$V=\{v=(z_1,\ldots,z_6)\in \C^6/\sum (-1)^i z_i=0\}/\C$
where $\C$ acts on $\C^6$ by diagonal translation,
and define on $V$ a Hermitian form $h$ as follows:
\begin{eqnarray*}
h(v,v)
&=&\operatorname{Area}(z_1,z_2,z_3)+\operatorname{Area}(z_4,z_5,z_6)\\
&=&\frac{1}{2}\det(z_2-z_1,z_3-z_1)+\frac{1}{2}\det(z_5-z_4,z_6-z_4)\\
&=&\frac{1}{2}\sum_{1\le i<j\le 6}(-1)^{i+j+1}\det(z_i,z_j)=
\frac{1}{2}\sum_{1\le i<j\le 6}(-1)^{i+j}\im(z_i\overline{z_j})\\
&=&\frac{1}{4i}\sum_{i<j} (-1)^{i+j}
(z_i\overline{z_j}-\overline{z_i}z_j).
\end{eqnarray*}
Also, put $q(v)=h(v,v)$ the underlying real quadratic form
and set $C=q^{-1}(0)$. The subset of non-aligned sextuples
in $C$ will be denoted by $C^\times$. With this notation,
$X_{\mathcal E}=C/S^1$ where $S^1$ acts diagonally
on $V$.
A simple computation shows that $h$ has signature $(2,2)$ on $V$;
in particular it is non-degenerate and its imaginary part gives a
symplectic form
on the real vector space
$V$, such that $q$ is a moment map for the diagonal action of $S^1$.
This implies that $X_{\mathcal E}$ has a natural symplectic form,
being a symplectic quotient; see \cite{McDuffSalamon}, Section 5.1 for
a reminder.

This symplectic structure has the property that the Hamiltonian flow of the length function
$d(z_1,z_2)$ (for instance) is the transformation fixing
$z_3,z_4,z_5,z_6$ and translating $z_1$ and $z_2$ along the
line joining them.  Moreover, the Hamiltonian flow of the function
$\operatorname{Area}(z_1,z_2,z_3)$ is given by
\[\Psi^t_{123}(z_1,\ldots,z_6)=(R_tz_1,R_tz_2,R_tz_3,z_4,z_5,z_6)\]
where $R_t$ is the rotation around the point $z_1-z_2+z_3$
(see Section \ref{SubsecTroiSym}).

\subsection{The Zariski tangent space}\label{zariski}

From now on in this section,
it will be more convenient to replace $\PSL$ by its isomorphic group
$ \PU(1,1)=
\left\lbrace\pm
\begin{pmatrix} a & b \\ \overline{b} & \overline{a}\end{pmatrix},
a,b\in\C, |a|^2-|b|^2=1\right\rbrace$
acting by homographies on the unit disc.
At a representation $\rho$, the Zariski tangent space to
$\Hom(\Gamma_o,\PU(1,1))$ may be described as the space of paths
$\rho_t\colon\gamma\mapsto\exp(tu(\gamma))\rho(\gamma)$
which, at first order, keep being representations of $\Gamma_o$.
This condition amounts to the relation
$u(\gamma_1\gamma_2)=u(\gamma_1)+\ad_\rho(\gamma_1)\cdot u(\gamma_2)$
for all $\gamma_1,\gamma_2$ where we have set $\ad_\rho(\gamma)\cdot \xi=\rho(\gamma)\xi\rho(\gamma)^{-1}$; the set of such maps $u$ is the
space $Z^1(\Gamma_o,\ad_\rho)$ of cocycles in group cohomology with
coefficients in $\pu(1,1)$ twisted by the adjoint action of $\rho$.
Coboundaries, of the form $\gamma\mapsto u_0-\ad_\rho(\gamma)u_0$,
correspond to (actual) deformations of $\rho$ by conjugation,
and the Zariski tangent space to $X(\Gamma_o)$ at a class $[\rho]$
is expected to be isomorphic to $H^1(\Gamma_o,\ad_\rho)$, see eg \cite{Weil64,Goldman84}.
This isomorphism holds at non-elementary representations, by
Proposition 5.2 of \cite{HeusenerPorti}.
In particular, if we denote by the same letter $\rho$ a
representation of $\Gamma_o$ and the corresponding representation of $\Gamma$,
then Proposition \ref{correspondence} implies that
$H^1(\Gamma_o,\ad_\rho)\simeq H^1(\Gamma,\ad_\rho)$ at every
non-elementary representation.

At singular representations it may happen however that the Zariski tangent
is not isomorphic to this cohomology group.
We do not investigate this question here, as we are not concerned with the
algebraic structure of $X(\Gamma_o)$.

At the singular representation, this cohomology group has a simpler description
as we explain here.
Define $\rho_0\colon\Gamma_o\rightarrow\PU(1,1)$ by
$\rho(c_i)=s_0=\pm\begin{pmatrix} i & 0\\ 0 & -i\end{pmatrix}$,
the half-turn around $0$, for $i=1,\ldots,6$.
In the natural decomposition
$\pu(1,1)=
\left\lbrace
\begin{pmatrix} ix & z \\ \overline{z} & -ix\end{pmatrix},
x\in\R,z\in\C\right\rbrace\simeq\R\oplus\C$, the element $\ad(s_0)$
acts trivialy on $\R$ and by multiplication by $-1$ on $\C$, so the
cocycle condition implies (with $\gamma_1=\gamma_2=c_i$ for all $i$)
that every cocycle has a trivial $\R$-part, and (with six terms) that
$\sum_i(-1)^i u(c_i)=0$, whereas a coboundary sends each $c_i$ to the same
complex number $2u_0$; this gives the natural identification
$V\simeq H^1(\Gamma_o,\ad_{\rho_0})\simeq H^1(\Gamma_o,\epsilon)$
where by $\epsilon$ we mean $\C$-coefficients twisted by the action
$\epsilon(c_i)z=-z$ for $i=1,\ldots,6$.
This is consistent in idea with the beginning of Paragraph \ref{euclidean}.
Indeed, the matrix
$\exp \begin{pmatrix} 0 & z\\ \overline{z} & 0 \end{pmatrix} s_0 =
\pm\begin{pmatrix} i\cosh(\rho) & -i\sinh(\rho)e^{i\theta}\\
i\sinh(\rho)e^{-i\theta}& -i\cosh(\rho) \end{pmatrix}$,
where $z=\rho e^{i\theta}$,
acts on the disc by a
half-turn around the point
$\tanh(\frac \rho 2)e^{i\theta}$.
Thus, if $u$ is an element of $H^1(\Gamma_o,\epsilon)$, the associated deformation
$\rho_t$ maps, at first order, $c_i$ to the half-turn around
$\frac{t}{2}u(c_i)$.

Now the inclusion map $\Gamma\rightarrow\Gamma_o$ yields a map
$Z^1(\Gamma_o,\epsilon)\rightarrow Z^1(\Gamma,\C)$ where the
$\C$-coefficients are not twisted any more since $\Gamma=\ker(\epsilon)$.
By checking on a basis, we will prove that this map
induces an isomorphism as follows.

\begin{proposition}
There is a natural isomorphism between $V$ and the cohomology
space $H^1(\Sigma,\C)$ such that
the Hermitian form $h$ corresponds
to the form $\frac{1}{4i} v\cdot \overline{w}$
where $\cdot$ denotes the cup product evaluated at the fundamental class. 
\end{proposition}

\begin{proof}
  Let $\gamma_1, \ldots, \gamma_6$ be as in Figure \ref{FigureNoms}
  and let $\gamma_i^\sharp\in H^1(\Sigma,\C)$ be the Poincar\'e
  duals of the corresponding cycles. From the signed intersections
  of the $\gamma_i$, we can read
  $\gamma_i^\sharp\cdot\gamma_{i+1}^\sharp=-1$,
  $\gamma_i^\sharp\cdot\gamma_{i-1}^\sharp=1$ for all $i$
  (with cyclic notation), and $\gamma_i^\sharp\cdot\gamma_j^\sharp=0$
  if $j\not\in\{i-1,i+1\}$~mod~6.
  By definition $\gamma_i^\sharp$ is the morphism mapping $\gamma_j$
  to $\gamma_j^\sharp\cdot\gamma_i^\sharp$; it is the image of
  the cocycle in $Z^1(\Gamma_o,\epsilon)$ mapping $c_i$ and $c_{i+1}$
  to $1$ and $c_j$ to $0$ for $j\neq i,i+1$.
  As the oriented curves $\gamma_1$, $\gamma_3$ and $\gamma_5$
  (resp. $\gamma_2$, $\gamma_4$ and $\gamma_6$) bound a subsurface,
  we have $\gamma_1^\sharp+\gamma_3^\sharp+\gamma_5^\sharp=0$ and
  $\gamma_2^\sharp+\gamma_4^\sharp+\gamma_6^\sharp=0$; these
  identities yield the kernel of the surjective morphism
  $f_1\colon\C^6\rightarrow H^1(\Sigma,\C)$ defined by
  $f_1(\lambda_1,\ldots,\lambda_6)=\sum_i\lambda_i\gamma_i^\sharp$.
\begin{figure}[hbt]
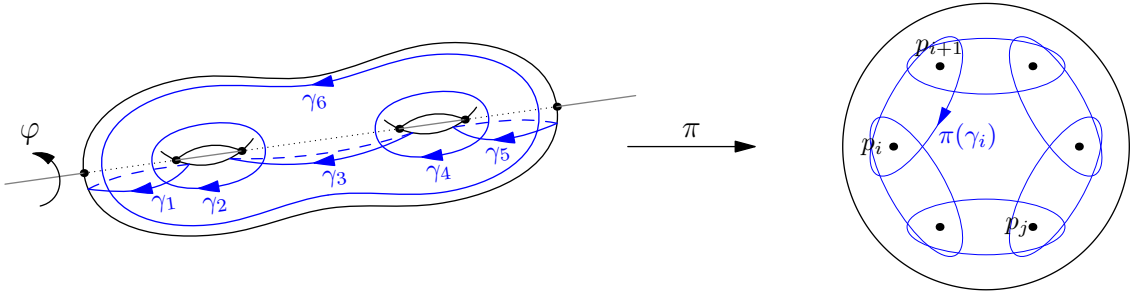

\begin{center}
\begin{asy}
  import geometry;

  picture Sigma;
  //
  //
  point p1=(0,0), p2=(35,0), p3=(60, 0), p4=(120, 0), p5=(145,0), p6=(180,0);

  path contourbas = p1{down}..(50,-30){right}..(90,-26){right}..(130,-30){right}..(180,0){up};
  path contourhaut= p1{up}..(45,30){right}..(90,24){right}..(135,30){right}..(180,0){down};

  dot (Sigma,p1); dot(Sigma,p2); dot(Sigma,p3);
  dot(Sigma,p4); dot(Sigma,p5); dot(Sigma,p6);
  draw(Sigma,contourbas); draw(Sigma,contourhaut);

  //
  path hautroug = p2{dir(33)}..p3{dir(-33)};
  path batroug = (p2+(-5,4))..p2..p3..(p3+(5,4));
  draw(Sigma,hautroug); draw(Sigma,batroug);
  path hautroud = p4{dir(33)}..p5{dir(-33)};
  path batroud = (p4+(-5,4))..p4..p5..(p5+(5,4));
  draw(Sigma,hautroud); draw(Sigma,batroud);

  //
  point m1=relpoint(contourbas,0.03), m2=relpoint(batroug,0.3);
  path gamma1=m1{dir(-15)}..m2{dir(45)}, gamma12=m2{left}..m1{dir(200)};
  draw(Sigma,reverse(gamma1),blue,Arrow(Relative(0.6)));
  label(Sigma,"\footnotesize $\gamma_1$",relpoint(gamma1,0.6),SSE,blue);
  draw(Sigma,gamma12,blue+dashed);
  point m3=p2+(-9,0), m4=0.5*(p2+p3)+(0,-12), m5=p3+(9,0), m6=m4+(0,24);
  path gamma2=m3{down}..m4{right}..m5{up}..m6{left}..cycle;
  draw(Sigma,reverse(gamma2),blue,Arrow(Relative(0.8)));
  label(Sigma,"\footnotesize $\gamma_2$",relpoint(gamma2,0.25),S,blue);
  point m7=relpoint(batroug,0.7), m8=relpoint(batroud,0.3);
  path gamma3=m7{dir(-20)}..m8{dir(20)}, gamma31=m7{dir(-10)}..m8{dir(10)};
  draw(Sigma,reverse(gamma3),blue,Arrow(Relative(0.55)));
  label(Sigma,"\footnotesize $\gamma_3$",relpoint(gamma3,0.55),S,blue);
  draw(Sigma,gamma31,blue+dashed);
  path gamma4=shift(p4-p2)*gamma2;
  draw(Sigma,reverse(gamma4),blue,Arrow(Relative(0.8)));
  label(Sigma,"\footnotesize $\gamma_4$",relpoint(gamma4,0.25),S,blue);
  point dd1=0.5*(p3+p4)+(0,10), dd2=dd1+(0,-10);
  path gamma51=reflect(dd1,dd2)*reverse(gamma1);
  draw(Sigma,reverse(gamma51),blue,Arrow(Relative(0.6)));
  label(Sigma,"\footnotesize $\gamma_5$",relpoint(gamma51,0.6),SW,blue);
  draw(Sigma,reflect(dd1,dd2)*gamma12,blue+dashed);
  point m9=p1+(7,0), m10=0.5*(p2+p3)+(0,-26), m11=0.5*(p3+p4)+(0,-21);
  path gamma6=m9{down}..m10{right}..m11{right};
  path gamma62=reflect(dd1,dd2)*gamma6, gamma63=reflect(p1,p2)*gamma62;
  path gamma64=reflect(p1,p2)*gamma6;
  draw(Sigma,gamma6,blue); draw(Sigma,gamma62,blue);
  draw(Sigma,gamma63,blue,Arrow); draw(Sigma,gamma64,blue);
  label(Sigma,"\footnotesize $\gamma_6$",reflect(p1,p2)*m11,S,blue);

  //
  draw(Sigma,(-30,0)--p1,grey); draw(Sigma,p1--p2,dotted);
  draw(Sigma,p2--p3,grey); draw(Sigma,p3--p4,dotted);
  draw(Sigma,p4--p5,grey); draw(Sigma,p5--p6,dotted);
  draw(Sigma,p6--(210,0),grey);
  //
  draw(Sigma,(-18,-10){dir(5)}..(-18,10){dir(170)},Arrow);
  label(Sigma,"$\varphi$",(-18,10),N);

  add(rotate(8)*Sigma,(-20,0));
  draw((185,10)--(233,10),Arrow); label("$\pi$",((185+233)/2,10),N);
  //
  picture sphere;
  point base = (0,0); real r=35; real R=54; point P1=r*dir(0), P2=r*dir(60);
  point P3=r*dir(120), P4=r*dir(180), P5=r*dir(240), P6=r*dir(300);
  path cercle=circle(base,R);
  path Cerc1=shift(0.5*(P2+P3))*yscale(0.6)*xscale(1.7)*circle((0,0),r/2);
  path Cerc2=rotate(60,(0,0))*Cerc1; path Cerc3=rotate(120,(0,0))*Cerc1;
  path Cerc4=rotate(180,(0,0))*Cerc1; path Cerc5=rotate(240,(0,0))*Cerc1;
  path Cerc6=rotate(300,(0,0))*Cerc1;

  draw(sphere,cercle);
  draw(sphere,Cerc1,0.2pt+blue); draw(sphere,Cerc2,0.2pt+blue);
  add(sphere,arrow(reverse(Cerc2),invisible,FillDraw(blue),Relative(0.27)));
  draw(sphere,Cerc3,0.2pt+blue); draw(sphere,Cerc4,0.2pt+blue);
  draw(sphere,Cerc5,0.2pt+blue); draw(sphere,Cerc6,0.2pt+blue);
  dot(sphere,P1); dot(sphere,P2); dot(sphere,P3);
  dot(sphere,P4); dot(sphere,P5); dot(sphere,P6);
  label(sphere,"\footnotesize $p_j$",P6,0.3*W);
  label(sphere,"\footnotesize $p_i$",P4,W);
  label(sphere,"\footnotesize $p_{i+1}$",P3,N);
  label(sphere,"\footnotesize $\pi(\gamma_i)$",0.13*(P3+P4),blue);
  add(sphere,(320,10));
\end{asy}
\end{center}
\caption{Cycles on the sphere and their lifts}
\label{FigureNoms}
\end{figure}
  The map $f_2\colon\C^6\rightarrow V$ defined by
  $f_2(\lambda_1,\ldots,\lambda_6)=
  [(\lambda_1+\lambda_2,\lambda_2+\lambda_3,\ldots,\lambda_6+\lambda_1)]$
  obviously vanishes on $\ker f_1$, hence $\ker f_1=\ker f_2$,
  so $f_1$ and $f_2$ induce an isomorphism
  $V\simeq H^1(\Sigma,\C)$.
  Straightforward computation gives, for
  $\lambda=(\lambda_1,\ldots,\lambda_6)$, that
  $f_1(\lambda)\cdot f_1(\overline{\lambda})=
  -2i\sum_i \im \lambda_i\overline{\lambda_{i+1}}$ and
  $q(f_2(\lambda))=
  -\frac{1}{2}\sum_i \im \lambda_i\overline{\lambda_{i+1}}$,
  so the two Hermitian forms coincide.
\end{proof}

\begin{remark}
\begin{itemize}
\item[-] Viewed as a subgroup of $\Aut(\Gamma_o)$ as we presented it in
Section \ref{SubsecMarkings}, our leapfrog group has an obvious action
on $V$. However, viewed as $B_6(S^2)\simeq\mcg(\Sigma)/[\varphi]$,
it acts on $X_0(\Gamma_o)$ but
not on the Zariski tangent space at $[\rho_0]$, as the hyperelliptic
involution acts by $-1$; this is an interesting subtlety.
\item[-] Let
$\xi\in H^1(\Gamma_o,\ad_{\rho_0})$ represent a tangent vector in $V$.
The cup-bracket
$[\xi\wedge\xi]\in H^2(\Gamma_o,\ad_{\rho_0})=
H^0(\Gamma_o,\ad_{\rho_0})^*=H^0(\Gamma_o,\R)^*=\R$
is an obstruction to $\xi$ being the tangent vector of a deformation of the
singular representation. This is a very natural description of the quadratic
form $q$ which will not be used in this article.
\end{itemize}
\end{remark}
For the simplicity of the exposition, let us explicit the symplectic form
mentioned above.
At a point $v\neq 0$ in $C$, the tangent space to $C$ is
$\lbrace w\in V\,|\,\re h(v,w)=0\rbrace$, so in $X_{\mathcal E}$,
the tangent space at $[v]=S^1 v$ is the quotient
\[T_{[v]}X_{\mathcal E}=\lbrace w\in V\,|\,\re h(v,w)=0\rbrace/\R iv;\]
now if $[w_1]$, $[w_2]\in T_{[v]}X_{\mathcal E}$ the formula
$\im h(w_1,w_2)$ is well-defined. In terms of the identification above
with $H^1(\Gamma_o,\ad_{\rho_0})$, the vector
$w_i=(\lambda_1^i,\ldots,\lambda_6^i)$ corresponds to the map
$\zeta_i\colon\Gamma\rightarrow\pu(1,1)$,
$\gamma_j\mapsto\left(\begin{array}{cc}0 & \lambda_j^i\\
\overline{\lambda_j^i} & 0\end{array}\right)$. Now the cup-product
of $\zeta_1$ and $\zeta_2$ composed with the Killing form can be computed
with the formula
$\tr \begin{pmatrix} 0 & z \\
\overline{z} & 0\end{pmatrix}\begin{pmatrix} 0 & w \\\overline{w} & 0
\end{pmatrix}=2\re z\overline{w}$,
giving $\tr(\zeta_1\cdot\zeta_2)=-8\im h(w_1,w_2)$.

\subsection{The Morse Lemma}
In this section, we prove that there is a neighbourhood of $[z_0]$ in
$X_0$ which is homeomorphic to $X_\mathcal{E}$ and diffeomorphic out of the singular configuration.
In the identification $V=H^1(\Sigma,\C)=
\Hom(H_1(\Sigma,\R),\C)$,
the subspace $C^\croix$ corresponds to
surjective maps $u:H_1(\Sigma,\R)\to \C$.
We have the
following proposition.
\begin{proposition}\label{morse}
There exist
neigbourhoods $U$ of $0$ in $C$, and $W$ of $\rho_0$ in $\hom(\Gamma_o,\PU(1,1))$
which are $S^1$-invariant,
and an $S^1$-equivariant homeomorphism
$f:U\to W$ which 
\begin{enumerate}
\item is a diffeomorphism on $U\setminus\{0\}$ to $W\setminus\{\rho_0\}$,
\item induces a diffeomorphism between $(U\setminus\{0\})/S^1$ and
$(W\setminus\{\rho_0\})/S^1$,
\item maps bijectively $C^\croix\cap U$ to $\hom^\croix(\Gamma_o,\PU(1,1))\cap W$.
\end{enumerate}
\end{proposition}

\begin{proof}
In this proof, we denote $s_0=\begin{pmatrix} i & 0 \\ 0 & -i\end{pmatrix}$ and for any $z_1,\ldots,z_6\in \C$ we set $\xi_j=\begin{pmatrix} 0 & z_j \\
\overline{z}_j & 0\end{pmatrix}$. We will look for a representation $\rho\in
\Hom(\Gamma_o,\PU(1,1))$ such that $\rho(c_j)=\pm\exp(\xi_j)s_0$.
For that reason,
we define in a neighbourhood of $0$ a map $F:\C^6\to \C$ and a map
$\varphi:\C^6\to \R$ by the formula
\[\exp(\xi_6)s_0\cdots \exp(\xi_1)s_0=
-\exp \begin{pmatrix} i\varphi & F\\\overline{F} & -i\varphi \end{pmatrix}.\]
We observe that conjugating the equation by the matrix $\begin{pmatrix}
e^{i\theta} & 0 \\ 0 & e^{-i\theta}\end{pmatrix}$ changes $z_j$ to
$e^{2i\theta}z_j$, $F$ to $e^{2i\theta}F$ and does not change $\varphi$.
Hence the map $\varphi$ is $S^1$-invariant and the map $F$ is $S^1$-equivariant. 

We have $-\exp(\xi_6)\cdots s_0\exp(\xi_1)s_0=\exp(\xi_6)\exp(-\xi_5)
\exp(\xi_4)\exp(-\xi_3)\exp(\xi_2)\exp(-\xi_1)$ and its logarithm is
$\sum_i (-1)^{i}\xi_i +\sum_{j>i}\frac{1}{2}[(-1)^j\xi_j,(-1)^i\xi_i]$ up
to order 2 terms thanks to the Baker-Campbell-Hausdorff formula. Hence, the Taylor
expansion gives $F(z_1,\ldots,z_6)
=\sum_{i=1}^6 (-1)^{i} z_i+o(|z|)$ and $\varphi(z_1,\ldots,z_6)=-2q(z)+o(|z|^2)$.
 
Consider the map $H:\C^6\to \C^6$ given by $H(z_1,\ldots,z_6)=
(z_1,\ldots,z_5,F(z_1,\ldots,z_6))$. By the inverse function theorem,
this is a local diffeomorphism. We observe that by construction, the
map $F$ is $S^1$-equivariant, hence the map $H$ and its inverse
are also $S^1$-equivariant. For small enough $w's$, write $\psi(w_2,\ldots,w_5)
=\varphi(H^{-1}(0,w_2,w_3,w_4,w_5,0))$. The $S^1$-invariant function
$\psi:\C^4\to \R$ has a non-degenerate Hessian at $0$ and vanish
identically if the $w_i's$ are real. We conclude by applying the Morse
Lemma \ref{lemme_morse}. Indeed, if we denote by $\Phi$ the diffeomorphism
provided by the lemma, we simply set $f(w_2,w_3,w_4,w_5)=H^{-1}(0,x_2,x_3,x_4,x_5,0)$
where $(x_2,x_3,x_4,x_5)=\Phi(w_2,w_3,w_4,w_5)$. 

The sextuples $(z_1,\ldots,z_6)\in \R^6$ correspond to linear maps
$v:H_1(\Sigma,\R)\to \C$ with values in $\R$. Hence, the $S^1$-orbit
of real configurations correspond precisely to linear maps of rank 0 or 1
and the diffeomorphism $f$ preserve aligned configurations as expected.  
\end{proof}
Observe that in the above computation, $q$ appears as a second
order obstruction for a
cocycle from being realized by deformations of representations; this is
yet another language for understanding this quadratic form.

\begin{lemma}[Equivariant Morse Lemma]\label{lemme_morse}
Let $\varphi:\C^n\to \R$ be a smooth $S^1$-invariant function, vanishing
on $\R^n$ and such that $\varphi(z)=Q(z)+o(|z|^2)$ for a non-degenerate
Hermitian form $Q$. Then there exist $S^1$-invariant neighbourhoods $U$ and $V$
of $0$ in $\C^n$ and an $S^1$-equivariant
diffeomorphism $\Phi:U\to V$ such that 
\begin{itemize}
\item[-] $D_0\Phi$ is the identity,
\item[-] $\Phi$ preserves $\R^n$,
\item[-] $\varphi\circ \Phi(z)=Q(z)$ for $z\in U$.
\end{itemize}
\end{lemma}
\begin{proof}
This is a variation of the standard Morse Lemma with the same proof,
using Moser's trick, see \cite{Lafontaine}, Theorem 3.44.
It is sufficient to check that the solution provided
by the proof has the properties required by the lemma.
\end{proof}

We observe that the space $X^\croix(\Gamma_o)$ as a subspace of
$X(\Sigma)$ is endowed with the Atiyah-Bott symplectic structure denoted
by $\omega_{AB}$. On the other hand, we explained that $X_{\mathcal E}$
is also symplectic with a symplectic structure denoted by $\omega$. We
do not know whether one can make the local diffeomorphism $f$ symplectic
but we will at least need the following weaker statement:

\begin{lemma}
The map $f$ constructed in Proposition \ref{morse} is a symplectomorphism
at first order by which we mean that the following holds:
\[(f^*\omega_{AB})_v=-8\omega+o(v).\]
\end{lemma}
\begin{proof}
Following Goldman (see \cite{Goldman84}), the Atiyah-Bott structure at $[\rho]\in X^\croix(\Sigma)$
is induced by the cup-product on $H^1(\Sigma,\ad_\rho)$ followed by the trace. 
The claimed approximation follows from Proposition \ref{morse}
and the computation ending Subsection \ref{zariski}.
\end{proof}

\subsection{The Grassmannian of Lagrangians}\label{lagrangian}

There is yet another description of the space $X_\mathcal{E}^\croix$ of
non-aligned Euclidean configurations which will be crucial in the last step
of the proof of Theorem \ref{SexErgodique}.
It uses the Grassmannian of Lagrangians in $H_1(\Sigma,\R)$
denoted by $\mathcal{L}$. 
It is a 3-dimensional manifold and the tangent space at $L\subset
H_1(\Sigma,\R)$ is canonically
isomorphic to the space of quadratic forms on $L$.
Indeed, the tangent space at $L$ to the Grassmannian of $2$-planes is
canonically isomorphic to the space
$\Hom(L,H_1(\Sigma,\R)/L)\simeq\Hom(L,L^*)$, where we identify
$H_1(\Sigma,\R)/L$ with $L^*$ by the symplectic pairing, and the Lagrangian
condition amounts to the symmetry of the corresponding bilinear maps.
Dually, the cotangent space of $\mathcal{L}$ at $L$ is isomorphic
to the space of quadratic forms on the dual space $L^*$. We denote by
$T^*_+\mathcal{L}\subset T^*\mathcal{L}$ the set of pairs $(L,\alpha)$
where $\alpha$ is a positive definite quadratic form on $L^*$.

\begin{proposition}
There is a diffeomorphism $\Lambda:X^\croix_{\mathcal{E}}\to
T^*_+\mathcal{L}$ which is equivariant with respect to the action of
$\Sp(H_1(\Sigma,\R))$. 
\end{proposition}
\begin{proof}
Recall that an element of $X^\croix_{\mathcal{E}}$ is the $S^1$-orbit
of a surjective linear map $u:H_1(\Sigma,\R)\to\C$ satisfying $q(u)=0$. 
Set $L=\ker u$ and show that $q(u)=0$ if and only if $L$ is Lagrangian. 

The quantity $q(u)$ is computed from any symplectic basis
$a_1,b_1,a_2,b_2$ of $H_1(\Sigma,\R)$ by the formula
$q(u)=\frac{1}{4}\im(u(a_1)\overline{u(b_1)}+u(a_2)\overline{u(b_2)})$.
If $L$ is Lagrangian, we can ensure that $a_1,a_2$ form a basis of
$L$ and hence $u(a_1)=u(a_2)=0$ and $q(u)=0$. If $L$ is not
Lagrangian, it is symplectic and we can form a basis of $L$ with
$a_1$ and $b_1$, which implies that $u(a_2)$ and $u(b_2)$ are
linearly independent and hence $q(u)\ne 0$. 

Hence, given $u:H_1(\Sigma,\R)\to\C$ with $q(u)=0$ and $L$ its
(Lagrangian) kernel, the expression $\alpha(x)=|u(x)|^2$ is a positive
definite quadratic form on $H_1(\Sigma,\R)/L\simeq L^*$, hence
$\alpha$ belongs to $T^*_+ L$ and does not change if we multiply
$u$ by a phase. This construction can be easily reversed and is
symplectically invariant
hence the map $\Lambda:u\mapsto (L,\alpha)$
has the required properties.
\end{proof}

\begin{lemma} The map $\Lambda$ is a symplectomorphism up to a constant.
\end{lemma}
\begin{proof}
Let $(L,g)$ be any point in $T^*_+\mathcal L$, one can find a symplectic
basis such that $L=\Span(a_1,a_2)$ and $a_1^*,a_2^*$ is an orthonormal
basis of $L^*$ with respect to $g$. The map $u=\Lambda^{-1}(L,g)$ is
given by $u(a_1)=u(a_2)=0, u(b_1)=1,u(b_2)=i$. 
There is a local coordinate system $(p_i,q_i)_{i=1,2,3}$ on $T^*\mathcal{L}$
given by setting $L_{p}=\R e_1\oplus\R e_2$ where $e_1=(1,0,p_1,p_2)$,
$e_2=(0,1,p_2,p_3)$ and $g_q$ has the matrix $\begin{pmatrix}
q_1 & q_2 \\ q_2 & q_3\end{pmatrix}$ in the basis $e_1^*,e_2^*$. In
that coordinate system, the symplectic form reads $\omega_{\mathcal L}
=\tr dg_p\wedge dg_q=dp_1\wedge dq_1+2 dp_2\wedge dq_2+dp_3\wedge dq_3$. 

The map $u_{p,q}=\Lambda^{-1}(L_p,g_q)$ is defined by sending
$e_1$ and $e_2$ to $0$ and $e_3,e_4$ to any basis $v_1^q,v_2^q$
of $\R^2$ whose Gram matrix is $g_q$. Explicitly one has 
\[u_{p,q}(a_1)=-p_1v_1^q-p_2v_2^q,u_{p,q}(a_2)=-p_2v_1^q-
p_3v_2^q,u_{p,q}(b_1)=v_1^q,u_{p,q}(b_2)=v_2^q.\]
Writing as a vector the values taken on the symplectic basis we
get at $(0,0,0,1,0,1)$ the following derivatives:
\[\frac{\partial\Lambda^{-1}}{\partial p_1}=(-1,0,0,0),\frac{\partial
\Lambda^{-1}}{\partial p_2}=(-i,-1,0,0),\frac{\partial\Lambda^{-1}}
{\partial p_3}=(0,-i,0,0).\]
Using the formulas $\frac{\partial v_1^q}{\partial q_1}=\frac{1}{2},
\frac{\partial v_1^q}{\partial q_2}=\frac{\partial v_1^q}{\partial q_3}=0$
and $\frac{\partial v_2^q}{\partial q_1}=0,\frac{\partial v_2^q}
{\partial q_2}=1,\frac{\partial v_2^q}{\partial q_3}=\frac{i}{2}$ we get
\[\frac{\partial\Lambda^{-1}}{\partial q_1}=(0,0,\frac{1}{2},0),
\frac{\partial\Lambda^{-1}}{\partial q_2}=(0,0,0,1),\frac{\partial\Lambda^{-1}}
{\partial q_3}=(0,0,0,\frac{i}{2}).\]
On the other hand, the symplectic structure on $V$ reads 
\begin{equation}\label{symp_expression}
\omega_V(v,w)=-\langle v(a_1),w(b_1)\rangle +\langle w(a_1),
v(b_1)\rangle-\langle v(a_2),w(b_2)\rangle+\langle w(a_2),v(b_2)\rangle.
\end{equation}
By checking in the basis, we find $\Lambda^*\omega_V=\frac{1}{2}
\omega_{\mathcal L}$ as asserted.
\end{proof}


\subsection{Topology of hourglasses}\label{topsab2}

\begin{proposition}
The space $X_0$ is homeomorphic to $X_\mathcal{E}$. The homeomorphism
maps the configuration $[\rho_0]$ to $0$, and $C^\croix/S^1$ to $X_0^\croix$. 
In particular we have:
\begin{itemize}
\item[-] $X_0^*$ is connected, simply connected and satisfies $\pi_2(X_0^*)=\Z$. 
\item[-] $X_0^\croix$ is homotopically equivalent to $\mathcal{L}$. Hence it is
connected and satisfies $\pi_1(X_0^\croix)=\pi_2(X_0^\croix)=\Z$.
\end{itemize}
\end{proposition}
\begin{proof}
By Proposition \ref{morse}, a punctured neighbourhood of $\rho_0$ is diffeomorphic
to $X_\mathcal{E}^*$.
As we saw in Section \ref{topsab1}, $X_0^*$ is diffeomorphic
to the set $f^{-1}(1,1+\varepsilon)$ for any $\varepsilon$ and the result follows. 
The Hermitian form $h$ has signature $(2,2)$ hence in some coordinates one has
$q(z_1,w_1,z_2,w_2)=|z_1|^2+|z_2|^2-|w_1|^2-|w_2|^2$. Normalizing the non-zero vectors
$(z_1,z_2)$ and $(w_1,w_2)$ we get the homeomorphism  
 $C^*\simeq S^3\times S^3\times \R$ and hence the homotopy equivalence $X_0^*\simeq S^3\times
S^3/S^1$. The long exact sequence of the fibration
$S^3\times S^3\rightarrow (S^3\times S^3)/S^1$ gives the fundamental
groups of $X_0^*$ claimed in the proposition.

Using the canonical Riemannian metric on $X_0^*$, the gradient flow of
$f$ preserves the set of sextuples which are on the same line. Hence the
same argument as above works for the space $X_0^\croix$ which is
homeomorphic to $C^\croix/S^1$. This latter space is homeomorphic to
$T_+^*\mathcal{L}$ which is a fiber bundle over $\mathcal{L}$ with convex
fiber. Finally we have the homotopy equivalence $X_0^\croix\simeq \mathcal{L}$. 
Considering the two-fold covering of $\mathcal{L}$ consisting of oriented Lagrangians in
$H_1(\Sigma,\R)$, we get a space $\tilde{\mathcal{L}}=U(2)/SO_2$, see \cite{McDuffSalamon}, Section 2.3.
The exact sequence of this fibration gives $\pi_1(\tilde{\mathcal{L}})
\simeq\pi_2(\tilde{\mathcal{L}})\simeq\Z$ and the same is true for $\mathcal{L}$. 
\end{proof}

\section{Dynamics}\label{SectionDyn}
\subsection{The strategy}
As in all proofs of ergodicity of mapping group actions on representation
spaces, we will use the periodicity properties of Goldman twist flows.
These flows are particularly simple to describe in the framework of sextuples. 
Let $z=(x_1,\ldots,x_6)\in\Sex_0$ be a sextuple and suppose that
$s_3s_2s_1$ is an elliptic element, that is a rotation over a point $y$.
Denote by $R_t$ the rotation of angle $t$ over $y$ and pick $\theta$
such that $s_3s_2s_1=R_\theta$. Then the formula 
\[\Phi_{123}^t z=(R_t x_1,R_t x_2,R_t x_3,x_4,x_5,x_6)\]
 defines a $2\pi$-periodic flow on $X_0$ such that 
$\Phi_{123}^\theta$ is the half-twist around the first three points.
A $\mcg(S_o)$-invariant function on $X_0$ is almost
 everywhere constant along the flow as the angle $\theta$ is almost
 everywhere irrational. However this argument works only when
 $s_3s_2s_1$ is elliptic. 

All such flows are indexed by {\it partition curves}, that is simple
curves $\gamma$ in $S_p$, the sphere minus the six marked points,
which divide the set of points into two
subsets of cardinality $3$. Indeed, for any $[\rho]\in X_0(\Gamma_o)$
we set $\Theta_\gamma([\rho])\in\mathbb{R}/2\pi\mathbb{Z}$ to be the
rotation number of $\rho(\gamma)$ (that is its angle if it is a rotation and
0 otherwise). Where $\Theta_\gamma$ is smooth, we define $X_\gamma$
to be the symplectic gradient of $\Theta_\gamma$ and extend it by $0$
where it is not defined. We denote by $\Phi_\gamma^t$ the flow of
$X_\gamma$. This definition is coherent in the sense that $\Phi_{123}
=\Phi_\gamma$ for a standard partition curve enclosing the 3 first points. 

\begin{remark} Let $\tilde{\gamma}$ be the preimage of $\gamma$ in
the surface $\Sigma$: then $\Phi_\gamma^t$ is the Goldman flow on
$X(\Sigma)$ associated to the separating curve $\tilde{\gamma}$ provided
that $\rho(\tilde{\gamma})$ is elliptic.
\end{remark}

\begin{definition}
For any $z\in X_0$ we set 
\[ \mathcal{D}_z=\Span\{X_\gamma(z), \gamma\text{ partition curve}\}
\subset T_{z}X_0.\]
\end{definition}
The aim of this section is to prove the following proposition:
 
\begin{proposition}\label{prop_nonintegrable}
The distribution $\mathcal{D}$ is completely non-integrable on the subset
$U\subset X_0$ of non-pinched configurations. 
\end{proposition}

By construction, the distribution is $\mcg(\Gamma_o)$-invariant, and by
Theorem \ref{SexMinimal}, the $\mcg(\Gamma_o)$-orbit of any point
in $U$ is adherent to the singular
representation $[\rho_0]$.
Hence it is sufficient to prove Proposition \ref{prop_nonintegrable} by
replacing $U$ with any neighbourhood of $[\rho_0]$. This will be a significant
simplification for two reasons:
\begin{enumerate}
\item For any partition curve $\gamma$, $\rho(\gamma)$ is elliptic for
$\rho$ close enough to $\rho_0$,
hence the vector fields $X_\gamma$ will not vanish close to the
singular representation.
\item Being completely non-integrable is an open condition, hence, we can
replace the functions $\Theta_\gamma$ by their Taylor expansion around
$[\rho_0]$ and reduce the problem to (symplectic) linear algebra. 
\end{enumerate}

The proof of Proposition \ref{prop_nonintegrable} is decomposed into two
subsections: in Subsection \ref{taylor} we compute the Taylor expansion of
the trace function associated to a partition curve and in Subsection
\ref{generating}, we show that the derivatives of these trace functions
generate the cotangent space around $[\rho_0]$.

To conclude this subsection, we recall the argument showing that
Proposition \ref{prop_nonintegrable} implies the ergodicity of $\mcg(\Gamma_o)$
on $X_0$. Let $f:X_0\to \R$ be a measurable invariant function.
Using standard ergodicity arguments (see Proposition 5.4 in
\cite{GoldmanXia11}), for any partition curve $\gamma$, there is a
measure $0$ subset $\mathcal{N}_\gamma\subset X_0\times \mathbb{R}$
such that $f(\Phi_\gamma^t(z))=f(z)$ for all $(z,t)\notin\mathcal{N}_\gamma$. 
Using Fubini theorem and the fact that the flows preserve nullsets,
for any partition curves $\gamma_1,\ldots,\gamma_n$, we will have
$f\circ\Phi_{\gamma_1}^{t_1}\circ\cdots\Phi_{\gamma_n}^{t_n}(z)=f(z)$
for almost all $(z,t_1,\ldots,t_n)\in X_0\times \mathbb{R}^n$.

Using cutoff functions one can smoothen the vector fields $X_\gamma$ without changing the
distribution $\mathcal{D}$ - hence we suppose that the vector fields are smooth from now.
Let $z\in X_0$ be a point in $U$. By the orbit theorem (see \cite{Jurdjevic}, Theorem 1 p.33), 
the orbit $N$ of $z$ through the action of the flows of $X_\gamma$ is a submanifold. By Proposition \ref{prop_nonintegrable}, 
the vector fields $X_\gamma$ and their brackets evaluated at $z$ generate $T_zX_0$. Hence $N$ is an 
open subset of $X_0$. Moreover the proof of the orbit theorem in \cite{Jurdjevic} shows that for any $z'\in N$ there 
exist $n\in\N$, curves $\gamma_1,\ldots,\gamma_n$ and $(t_1^0,\ldots,t^0_n)\in\R^n$ such that the map $F:\R^n\to N$ defined by 
\[F(t_1,\ldots,t_n)=\Phi_{\gamma_1}^{t_1}\circ\cdots \circ\Phi_{\gamma_n}^{t_n}(z)\]
satisfies $F(t^0_1,\ldots,t^0_n)=z'$ and $\operatorname{rank} DF(t^0_1,\ldots,t^0_n)=\dim N = 6$. 
From the fact that $f$ is almost constant in the image of $F$, we get that $f$ is almost everywhere constant in a neighborhood of $z'$, hence in $N$, a neighborhood of $z$. 

By Proposition \ref{prop_nonintegrable}, this
argument works for any point in the connected set $U$ which has full
measure, showing that $f$ is almost everywhere constant.

\begin{remark}
In the spirit of Section \ref{SectionAlgo} we can imagine a proof
by hand of the transitivity of the flows $\Phi_\gamma^t$. This proof
of Theorem \ref{SexErgodique} would be slightly more direct but less
informative about the structure of these hourglass representations,
and we chose not to develop it here.
\end{remark}

\subsection{Taylor expansion of trace functions}\label{taylor}
The set 
\[\widetilde{\Sex}=\{\rho:\pi_1(S_p)\to \SLdeuxR\textrm{ such that }
\tr\rho(c_i)=0\,\textrm{ for } i=1,\ldots,6\}.\]
yields a regular covering of $X_0$, which is contractible as we proved
in Section \ref{topsab1}. Thus, we may choose once for all a lift
$\tilde{\rho_0}$ of the singular representation, and lift every
representation $\rho$ accordingly.
With this setting, for any partition curve $\gamma$, we set $F_\gamma([\rho])=
\tr \tilde{\rho}(\gamma)$. This is a continuous function on $X_0$, smooth
on $X^\croix_0$ and which vanishes at $[\rho_0]$. Our purpose is to
compute its Taylor expansion at~$[\rho_0]$.

\begin{proposition}\label{prop_taylor}
For any partition curve $\gamma\subset S_p$, let
$\tilde{\gamma}\subset \Sigma$
be its (separating) pre-image in
$\Sigma$. Then $\Sigma$ can be written as $\Sigma'
\cup_{\tilde{\gamma}}\Sigma''$.
Write $\xi\in C$ as $\xi=\xi'+\xi''$ using the decomposition $H^1(\Sigma,\C)
=H^1(\Sigma',\C)\oplus H^1(\Sigma'',\C)$ and set $q_\gamma(\xi)=q(\xi')$,
then in the chart given by Proposition~\ref{morse} we have 
\[F_{\gamma}(\xi)=\pm 8 q_\gamma(\xi)+o(|\xi |^2).\]
\end{proposition}
\begin{proof}
Remark that as $\xi$ is in $C$, we have $q(\xi)=q(\xi')+q(\xi'')=0$, hence we can
replace $\xi'$ with $\xi''$ in this formula. 
Up to the action of $\mcg(\Sigma)$ we can suppose that $\gamma=c_1c_2c_3$.
The same direct computation as in the proof of
Proposition~\ref{morse} then gives
\begin{eqnarray*}
\tr \tilde{\rho}(c_1)\tilde{\rho}(c_2)\tilde{\rho}(c_3)
&=&\pm\tr(e^{\xi_1}e^{-\xi_2}e^{\xi_3}s_0)
=\pm\frac{1}{2}\tr s_0([\xi_1,\xi_3]-[\xi_1,\xi_2]-[\xi_2,\xi_3])+o(|\xi|^2)\\
&=&\pm\!\!\!\sum_{1\le j<k\le 3}\!\!(-1)^{j+k}(z_j\overline{z}_k-z_k\overline{z}_j)
+o(|\xi|^2)
=\pm 8 q(z_1,z_2,z_3,0,0,0)+o(|\xi|^2),
\end{eqnarray*}
with the same notation.
\end{proof}

The computation here relates the splitting of $q$ into two
terms with the orthogonal decomposition of the cohomology space. We already saw 
this splitting in the first definition of $X_{\mathcal E}$ where the quadratic constraint appeared 
as a sum of two areas.

\subsection{Generating the cotangent space}\label{generating}

Let us go back to the proof of Proposition \ref{prop_nonintegrable}. We
recall that it amounts to proving that the Hamiltonian vector fields
$X_\gamma$ of the functions $F_\gamma$ generate a completely
non-integrable distribution close enough to the singular configuration.
Using Proposition \ref{prop_taylor} and the fact that being completely
non-integrable is an open condition, it reduces to proving the following
proposition:

\begin{proposition}\label{prop_generating}
For any partition curve $\gamma$, let  $Y_\gamma$ be the Hamiltonian
vector field of the function $q_\gamma$. For any $z\in X^\croix_\mathcal{E}$,
set 
\[E_z=\Span\{Y_\gamma(z),\gamma\textrm{ partition curve}\}\subset T_z
X^\croix_\mathcal{E}.\]
Then, $E$ is a completely non-integrable distribution on
$X^\croix_\mathcal{E}$.
\end{proposition}

\begin{proof}
First, we recall that through the symplectic isomorphism
$TX^\croix_\mathcal{E}\simeq T^*X^\croix_\mathcal{E}$,
the vector field $Y_\gamma$ corresponds to the
covector $dq_\gamma$. Hence, the symplectic orthogonal $E^\omega$
to the distribution $E$ has the following description: 
\[E^\omega_z=\{w\in T_z X_\mathcal{E}^\croix, D_zq_\gamma(w)=0, \forall \gamma
\text{ partition curve in }S_p\}.\]

We will use the model $X^\croix_\mathcal{E}=C^\croix/S^1$
where $C^\croix$ is the set of
surjective maps $v:H_1(\Sigma,\R)\to\C$ such that $q(v)=h(v,v)=0$,
with $h(v,w)=-\frac{i}{4}v\cdot\overline{w}$.
Recall that a vector $w$ is tangent to $C$ at $v\neq 0$ if
$\re h(v,w)=0$.

\begin{lemma}\label{identity}
Let $v\in C^\times$ and let $w\in V$ be tangent to $C$ at $v$ and satisfy
$Dq_\gamma(v)(w)=0$ for any partition curve $\gamma$. Then $v$ and $w$,
as linear maps from $H_1(\Sigma,\R)$ to $\R^2$, satisfy
the following equation:
\begin{equation}\label{relation}
\forall x,y\in H_1(\Sigma,\R),\quad \det(v(x),w(y))=\det(v(y),w(x)).
\end{equation}
\end{lemma}
\begin{proof}
First, observe that for any symplectic basis $a_1,b_1,a_2,b_2$ of
$H_1(\Sigma,\R)$ we have
\[-4\re h(v,w)=\det(v(a_1),w(b_1))-
\det(v(b_1),w(a_1))+\det(v(a_2),w(b_2))-\det(v(b_2),w(a_2)).\]

Now, for any partition curve $\gamma$, let
$p_\gamma \in\textrm{End}(H^1(\Sigma,\C))$ be the ($h$-orthogonal)
projection on $H^1(\Sigma',\C)$ parallel to $H^1(\Sigma'',\C)$;
note that $\re h(p_\gamma(\cdot),\cdot)$ is a symmetric bilinear form,
associated to~$q_\gamma$. Hence the condition
$Dq_\gamma(v)(w)$ is equivalent to $\re h(p_\gamma(v),w)=0$. If we
choose $x=a_1$ and $y=b_1$ this condition is equivalent to Equation
\eqref{relation}.

The same holds if we replace $(x,y)$ with its image by any transformation
in Sp$(4,\Z)$. So the map Sp$(4,\R)\to \R$ sending $g$ to
$\det(v(gx),w(gy))-\det(v(gy),w(gx))$ vanishes on Sp$(4,\Z)$.  By the
Zariski-density of $\Sp(4,\Z)$ in $\Sp(4,\R)$, this forces
Equation \eqref{relation} to hold for any
$x,y\in H_1(\Sigma,\R)$ such that $x\cdot y=1$. By scaling $x$ or $y$, this
holds finally for any $x$ and $y$ and the lemma is proved. 
\end{proof}

Fix $v:H_1(\Sigma,\R)\to\C$ surjective and satisfying $q(v)=0$. Recall from
Subsection \ref{lagrangian} that its kernel $L$ has to be Lagrangian. Let $w$
be in $E_v^\omega$ : Lemma \ref{identity} implies that $w$ vanishes on $L$.
Hence, $f=w\circ v^{-1}$ is a well defined endomorphism of $\R^2$ and writing
$x=v^{-1}(x')$ and $y=v^{-1}(y')$ we get $\det(f(x'),y')+\det(x',f(y'))=0$.
Hence, $f$ preserves infinitesimally the form $\det$, or equivalently has
vanishing trace. On the other hand, the equation $\re h(v,w)=0$ is automatically
satisfied in the preceding conditions. To sum up we have shown that the
orthogonal distribution $E_v^\omega$ is the following 2-dimensional space:
\[E_v^{\omega}=\{w\in \Hom(H_1(\Sigma,\R),\C), w|_{\ker v}=0,
\tr(w\circ v^{-1})=0\}/\R iv.\]

If $u\in \Hom(H_1(\Sigma,\R),\C)$ satisfies $u|_L=0$ then $\re h(u,v)=0$.
All such maps form a Lagrangian containing $E_v^\omega$. It follows that
$E_v$ contains this 3-dimensional space. This implies that one can express
that $u$ belongs to $E_v$ by looking at the restriction of $u$ to $L$. 

Let us show that one has the following description:
\begin{equation}\label{eq_distrib}
E_v=\{u\in \Hom(H_1(\Sigma,\R),\C)\text{ s.t. }\exists \lambda\in \R,\,
u|_L=\lambda (v^*)^{-1}\}.
\end{equation}
In this formula, we see $v$ as a map $H_1(\Sigma,\R)/L\to\C$ and identify
$H_1(\Sigma,\R)/L$ with $L^*$ via the intersection form. 
Consider an adapted symplectic basis $a_1,b_1,a_2,b_2$ of
$H_1(\Sigma,\R)$.
An element $w$ in $E_v^\omega$ vanishes on $a_1,a_2$
and its  matrix $M=(w(b_1),w(b_2))$ has trace $0$. Let $u$ be in $E_v$
and set $N=(u(a_1),u(a_2))$. From the expression of symplectic structure
given by equation \eqref{symp_expression}, we derive that the equation
$\tr MN=0$ should be satisfied for all $M$ with 0 trace. Hence $N$ is a
scalar matrix and the formula is proved. 

This defines a 4-dimensional distribution on $C^\croix/S^1$. It remains to show
that it is completely non-integrable which is the content of the last lemma. 
\end{proof}
\begin{lemma}
Let $\mathcal{L}$ be the Grassmanian of all Lagrangian subspaces of
$H_1(\Sigma,\R)$. Recall the identification $X_\mathcal{E}^\croix:C^\croix/S^1=
T^*_+\mathcal{L}$ and denote by $\pi:T^*\mathcal{L}\to\mathcal{L}$ the
standard projection.
 
Consider the exact sequence 
\[0\to T^*_L \mathcal{L}\to T_{(L,g^{-1})} T^*\mathcal{L}\overset{d\pi}\to
T_L\mathcal{L}\to 0.\]
The distribution $E$ corresponds to the distribution
$F_{(L,g^{-1})}=d\pi^{-1}(\R g)$  and is completely non-integrable.
\end{lemma}
\begin{proof}
We will show that the flows of the vector fields generated by elements of $E$
act transitively on $T^*_+\mathcal{L}$. By Frobenius theorem (see Theorem 4 
p. 45 in \cite{Jurdjevic}) and using the fact that $T^*_+\mathcal{L}$ is an homogeneous
space, this is equivalent to the complete non-integrability of $E$.

The identification between $E$ and $F$ is a simple transcription of
Equation \eqref{eq_distrib}: it remains to show that $F$ is completely
non-integrable.
Let $T_+\mathcal{L}\subset T\mathcal{L}$ be the cone of positive
directions. As $F$ contains all the directions of the fiber of $\pi$ one can
move from $(L,g^{-1})$ to all elements of the form $(L,(g')^{-1})$ for $g'$
close to $g$. Following the corresponding direction $g'$ in $\mathcal{L}$,
one can move from $L$ in $\mathcal{L}$ towards the direction of positive
quadratic forms, moving again along the fiber of $\pi$, the problem reduces
to show that the conical distribution $T_+\mathcal{L}$ is completely
non-integrable in $\mathcal{L}$ which is clear from the fact that definite
(i.e. positive or negative) quadratic forms linearly generate the space of
all quadratic forms. 
\end{proof}


\bibliographystyle{plain}

\bibliography{Biblio}

\end{document}